\documentclass[review]{elsarticle}

\usepackage{algorithm}
\usepackage{algorithmic}
\usepackage{graphicx}
\usepackage{epstopdf}
\usepackage{diagbox}
\usepackage{amsthm,amsfonts,amssymb}
\usepackage{multirow}
\usepackage{booktabs} 
\usepackage{float}
\usepackage{mathrsfs}
\usepackage{lipsum}

\usepackage{subfigure}
\usepackage[skip=0pt]{subcaption}



\usepackage{amsmath}
\allowdisplaybreaks[4]

\usepackage{amsmath}
\usepackage{amssymb}
\usepackage{url}
\usepackage{amsthm,amsfonts,amssymb}
\usepackage{caption}
\usepackage{subcaption}

\newtheorem{theorem}{Theorem}[section] 
\newtheorem{definition}[theorem]{Definition} 
\newtheorem{lemma}{Lemma}

\usepackage{lipsum}
\usepackage{amsfonts}
\usepackage{multirow}
\usepackage{graphicx}
\usepackage{epstopdf}
\usepackage{graphicx,epstopdf} 
\usepackage[caption=false]{subfig} 
\usepackage{algorithmic}
\ifpdf
  \DeclareGraphicsExtensions{.eps,.pdf,.png,.jpg}
\else
  \DeclareGraphicsExtensions{.eps}
\fi

\newcommand{\ii}{{\bf{i}}}
\newcommand{\jj}{{\bf{j}}}
\newcommand{\kk}{{\bf{k}}}

\usepackage{amsopn}

\journal{Journal of \LaTeX\ Templates}

\begin{document}

\begin{frontmatter}



    \title{Improved two-block coordinate descent method for Pose Graph Optimization Problem under $F^*$-norm\tnoteref{mytitlenote}}
    \tnotetext[mytitlenote]{The author Liping Zhang  is supported by the National Natural Science Foundation of China (Grant No. 12171271).}


    \author{Yongjun Chen}
    \author{Liping Zhang\corref{mycorrespondingauthor}}
    \cortext[mycorrespondingauthor]{Corresponding author. \emph{Email address}:lipingzhang@mail.tsinghua.edu.cn}
    \address{Department of Mathematical Sciences, Tsinghua University, Beijing 100084, China}

\begin{abstract}
Dual quaternions and dual quaternion matrices are widely used in robotics research, particularly in simultaneous localization and mapping (SLAM) problem. Using dual quaternion theory and graph-based methods, SLAM can be reformulated as a rank‑one dual quaternion Hermitian matrix completion problem, known as the pose graph optimization (PGO) problem. Recently, Qi and Cui introduced a two‑block coordinate descent method to solve this reformulated problem. In this paper, we enhance this method by reformulating the PGO problem under the more appropriate and robust $F^*$-norm rather than the conventional Frobenius norm, leading to improved experimental accuracy. We show that under the $F^*$-norm, one block has a closed-form solution and another is the optimal rank‑one approximation of dual quaternion Hermitian matrices under the $F^*$-norm. We derive an explicit solution for this 
approximation and present an efficient algorithm to compute it. To further enhance the two-block coordinate descent method, we introduce proper parameter selection, stagnation-based termination criteria and an effective spectral initialization strategy. Extensive numerical experiments demonstrate that our refinements deliver superior accuracy, faster computation, and higher success rates, particularly in low-observation settings. In particular, using the $F^*$‑norm outperforms the traditional $F$‑norm, underscoring its ability to more faithfully capture the magnitude of the dual parts of dual quaternion matrices.
\end{abstract}




\end{frontmatter}


\renewcommand{\arraystretch}{0.75}
\setlength\tabcolsep{3pt}
\captionsetup[table]{skip=8pt}

\section{Introduction}
Dual quaternions and dual quaternion matrices play an important role in robotics research and offer powerful tools to solve challenges such as the hand-eye calibration problem \cite{Daniilidis1999,hec6,hec1,hec5,hec3}, the simultaneous localization and mapping problem (SLAM) \cite{qc1, c1, Cadena2016, Grisetti2010}, and the kinematic modeling problem \cite{Qi2023}. Dual quaternion Hermitian matrices, a crucial subclass, exhibit remarkable properties, which are instrumental in analyzing the stability of multi-agent formation control systems \cite{Qi2023} and solving pose graph optimization (PGO) problems \cite{Cadena2016, cui2024power, Grisetti2010}.  

The simultaneous localization and mapping is a foundational technology in robotics and computer vision. It involves the concurrent estimation of a mobile robot or agent's position and orientation (localization) within an unknown environment while incrementally constructing a map of that environment. SLAM has been a prominent research focus in mobile robotics over the past two decades. An intuitive way to conceptualize SLAM is through a graph-based formulation \cite{lu1997globally}, where nodes represent the poses of the robot at various times, and edges encode constraints derived from environmental observations or robot movements \cite{Grisetti2010}. Under this transformation, the SLAM problem is also called the pose graph optimization (PGO) problem. In this framework, the map is computed by determining the most consistent spatial configuration of the nodes with respect to the edge-modeled measurements. Specifically, given a directed graph $G=(V,E)$, each vertex $i\in V$ corresponds to a robot pose $\hat{q}_i\in \hat{\mathbb U}$ and each directed edge $(i, j)\in E$ corresponds to a relative measurement $\hat{q}_{ij}\in \hat{\mathbb U}$. $\hat{\mathbb{U}}$ is denoted as the set of unit dual quaternions. The PGO problem then seeks to find the optimal $\hat{q}_i$, for all $i$, that minimizes the discrepancy between the relative measurements $\hat{q}_{ij}$ and the corresponding poses $\hat{q}_{i}$ and $\hat{q}_{j}$. A piece of work \cite{kummerle2011g} reformulates SLAM as a nonlinear least squares problem and solves it using the Gauss-Newton method. However, due to the non-convex nature of the problem, the Gauss-Newton method can become trapped in local minima, making a good initial estimate crucial \cite{carlone2015initialization}. For improved global convergence, \cite{fan2020majorization} introduces a majorization-minimization approach, which guarantees convergence to first-order critical points under mild conditions, providing a robust solution for distributed PGO. 

Recently, Qi and Cui \cite{cui2024power} reformulated the pose graph optimization problem as a rank-one dual quaternion completion task and proposed a two-block coordinate descent method to solve it. This method demonstrated promising numerical results.
Let $\hat{\mathbf{x}}=[\hat{q}_1,\ldots,\hat{q}_n]^*$ and $\hat{\mathbf{Q}}=(\hat{q}_{ij})$, than the pose graph optimization problem can be established as
\begin{equation}
\min_{\hat{\mathbf{x}}\in\hat{\mathbb{U}}^{n\times1}}\|P_E(\hat{\mathbf{Q}}-\hat{\mathbf{x}}\hat{\mathbf{x}}^*)\|_F^2.
\end{equation}
where, let $\hat{\mathbf{A}}=(\hat{a}_{ij})\in\hat{\mathbb{Q}}^{n\times n}$ and $P_E(\hat{\mathbf{A}})=(\hat{p}_{ij})$, then $\hat{p}_{ij}=\hat{a}_{ij}$ for $(i, j)\in E$, and $\hat{p}_{ij}=\hat{0}$, otherwise. \cite{cui2024power} then reformulate it as an equivalent problem:
\begin{equation}
\begin{aligned}
\min_{\hat{\mathbf{X}}_1,\hat{\mathbf{X}}_2}&
\frac12\|P_E(\hat{\mathbf{Q}}-\hat{\mathbf{X}}_1)\|_F^2 \\
\mathrm{s.t.}&~\hat{\mathbf{X}}_1\in C_1,\hat{\mathbf{X}}_2\in C_2,\hat{\mathbf{X}}_1=\hat{\mathbf{X}}_2,
\end{aligned}
\end{equation}
where $C_1=\{\hat{\mathbf{X}}\in\hat{\mathbb{U}}^{n\times n}|\hat{\mathbf{X}}=\hat{\mathbf{X}}^*,\hat{x}_{ii}=\hat{1}\}$,
$C_2=\{\hat{\mathbf{X}}\in\hat{\mathbb{Q}}^{n\times n}|\text{rank}(\hat{\mathbf{X}})=1\}$.

Applying the quadratic penalty approach, it can be reformulated as:
\begin{equation}
\begin{aligned}
\min_{\hat{\mathbf{X}}_1,\hat{\mathbf{X}}_2}&\frac12\|P_E(\hat{\mathbf{X}}_1-\hat{\mathbf{Q}})\|_F^2+\frac\rho2\|\hat{\mathbf{X}}_1-\hat{\mathbf{X}}_2\|_F^2,
\\\mathrm{s.t.}&~\hat{\mathbf{X}}_1\in C_1,\hat{\mathbf{X}}_2\in C_2.
\end{aligned}
\end{equation}
Then, a two-block coordinate descent method is proposed to tackle this reformulated problem, where $\hat{\mathbf{X}}_1^{(k)}$ and $\hat{\mathbf{X}}_2^{(k)}$ are obtained by solving the following subproblems in the $k$-th iteration, alternatively:
\begin{align}
\hat{\mathbf{X}}_1^{(k)}&=\underset{\hat{\mathbf{X}}_1\in C_1}{\operatorname*{argmin}}
\frac12\|P_E(\hat{\mathbf{X}}_1-\hat{\mathbf{Q}})\|_F^2+\frac{\rho^{(k-1)}}2\|\hat{\mathbf{X}}_1-\hat{\mathbf{X}}_2^{(k-1)}\|_F^2,\label{q1}\\
\hat{\mathbf{X}}_2^{(k)}&=\underset{\hat{\mathbf{X}}_2\in C_2}{\operatorname{argmin}} 
\frac{\rho^{(k-1)}}2\|\hat{\mathbf{X}}_2
-\hat{\mathbf{X}}_1^{(k)}\|_F^2.\label{q2}
\end{align}
Subproblem \eqref{q1} has an explicit solution
\begin{equation}\label{e42}
\hat{x}_{1,ij}^{(k)}=\begin{cases}
\hat{1} ,&\text{if}~~i=j,\\
P_{\hat{\mathrm{U}}}\left(\frac{\delta_{E,ij}
\hat{q}_{ij}+\delta_{E,ji}\hat{q}^*_{ji}+\rho^{(k-1)}(
\hat{x}_{2,ij}^{(k-1)}+(\hat{x}_{2,ji}^{(k-1)})^*)
}{\delta_{E,ij}
+\delta_{E,ji}+2\rho^{(k-1)}}\right),&\text{if}~~i\ne j,
\end{cases}
\end{equation}
where, $\delta_{E,ij}=1$ for $(i,j)\in E$ and $\delta_{E,ij}=0$, otherwise, $P_{\hat{\mathrm{U}}}(\cdot)$ is the projection onto the set of unit dual quaternions.
Subproblem \eqref{q2} admits an explicit solution
\begin{equation}\label{e44}
\hat{\mathbf{X}}_2^{(k)}=\hat{\lambda}\hat{\mathbf{u}}\hat{\mathbf{u}}^*,
\end{equation}
where $\hat{\lambda}$ and $\hat{\mathbf u}\in \hat{\mathbb U}^n_2$ are the dominate eigenpair of $\hat{\mathbf{X}}_1^{(k)}$. The detailed algorithm is presented in Algorithm \ref{PGOP0}.
\begin{algorithm}
    \caption{A two-block coordinate descent method for solving PGO problem}\label{PGOP0}
    \begin{algorithmic}
    \REQUIRE $\hat{\mathbf Q}$, $(V,E)$, initial iteration matrix $\hat{\mathbf{X}}_1^{(0)}$ and $\hat{\mathbf{X}}_2^{(0)}$, the
    maximum iteration number  $k_{max}$, parameter $\rho^{(0)}$, $\rho_1$, and the tolerance $\beta$.
    \ENSURE $\hat{\mathbf{X}}_1^{(k)}$, $\hat{\mathbf{X}}_2^{(k)}$.
    \FOR{$k=1,2,\cdots,k_{max}$}
    \STATE Update $\hat{\mathbf{X}}_1^{(k)}$ by \eqref{e42}.
    \STATE Update $\hat{\mathbf{X}}_2^{(k)}$ by \eqref{e44}.
    \STATE Update $\rho^{(k)}=\rho^{(k-1)}\times \rho_1$.
    \STATE If $\|\hat{\mathbf{X}}_1^{(k)}-\hat{\mathbf{X}}_2^{(k)}\|_{F^R}\le\beta$, then Stop.
    \ENDFOR
    \end{algorithmic}
\end{algorithm}

This work establishes a connection between dual quaternion matrix theory and the SLAM problem. The two-block coordinate descent method comprises two components: one block admits a closed‑form update, while the other requires solving the optimal rank-1 approximation of dual quaternion Hermitian matrices under the $F$-norm. The $F$-norm \cite{Qi2022} of a dual quaternion matrix $\hat{\mathbf Q}=\tilde{\mathbf Q}_{st}+\tilde{\mathbf Q}_{\mathcal I}\varepsilon$ is defined by
\begin{equation*}
\| \hat{\mathbf Q} \|_{F}=\begin{cases}
\| \tilde{\mathbf Q}_{st} \|_{F}+\frac{sc (tr(\tilde{\mathbf Q}_{st}^{\ast }\tilde{\mathbf Q}_{\mathcal I}))}{\|  \tilde{\mathbf Q}_{st} \|_{F}}\varepsilon, & \text{if} \quad \tilde{\mathbf Q}_{st}\neq \tilde{\mathbf O }, \\
 \| \tilde{\mathbf Q}_{\mathcal I} \|_{F}\varepsilon, & \text{if} \quad \tilde{\mathbf Q}_{st}=\tilde{\mathbf O },
\end{cases}
\end{equation*}
where $sc(\hat{q})=\frac{1}{2}(\hat{q}+\hat{q}^*)$ and $tr(\hat{\mathbf Q})$ denotes the sum of the diagonal elements of $\hat{\mathbf Q}$.  

Since the dual part of the $F$-norm of a dual quaternion matrix couples its standard and dual parts, a small overall $F$‑norm does not guarantee that the entries of the dual part are small. The following example illustrates this phenomenon.
Let
\begin{equation*}
\hat{\mathbf{Q}}=\tilde{\mathbf{Q}}_{st}+\tilde{\mathbf{Q}}_{\mathcal{I}}\varepsilon=
\begin{bmatrix}
10^{-5}   & 0\\
0 & 0
\end{bmatrix}+\begin{bmatrix}
0 & 0\\
0 & 10^{5} 
\end{bmatrix}\varepsilon.
\end{equation*}
Then the $F$-norm of $\hat{\mathbf{Q}}$ is given by $\|\hat{\mathbf{Q}}\|_{F} = 10^{-5}$ and $\|\hat{\mathbf{Q}}\|^2_{F} = 10^{-10}$, while an element $10^{5}$ exists in $\tilde{\mathbf{Q}}_{\mathcal{I}}$. This example demonstrates that the $F$‑norm can fail to reflect the true magnitude of the dual parts of dual quaternion matrices. To overcome this limitation, we instead adopt the $F^*$‑norm, which more faithfully measures the size of the dual part. The $F^*$-norm of a dual quaternion matrix $\hat{\mathbf Q}=\tilde{\mathbf Q}_{st}+\tilde{\mathbf Q}_{\mathcal I}\varepsilon$ is defined by
\begin{equation*}
\| \hat{\mathbf Q} \|_{F^*}=\begin{cases}
\| \tilde{\mathbf Q}_{st} \|_{F}+\frac{\|\tilde{\mathbf Q}_{\mathcal I}\|^2_F}{2\|\tilde{\mathbf Q}_{st} \|_{F}}\varepsilon, & \text{if} \quad \tilde{\mathbf Q}_{st}\neq \tilde{\mathbf O }, \\
 \| \tilde{\mathbf Q}_{\mathcal I} \|_{F}\varepsilon, & \text{if} \quad \tilde{\mathbf Q}_{st}=\tilde{\mathbf O }.
\end{cases}
\end{equation*}
Then the $F^*$-norm of the above example is $10^{-5}+\frac{1}{2}10^{15}\varepsilon$ and $\| \hat{\mathbf Q} \|^2_{F^*}=10^{-10}+10^{10}\varepsilon$, which more faithfully captures the magnitudes of the dual part of a dual quaternion matrix. Therefore, within the framework of the given two-block coordinate descent method, we aim to enhance the accuracy of the solution by formulating the PGO problem under the $F^*$-norm. And solving the subproblem motivates our study of the optimal rank‑$k$ approximations of dual quaternion Hermitian matrices under the $F^*$‑norm.

The dual complex adjoint matrix was introduced in \cite{chen2024dual}, creating a pivotal connection between dual quaternion matrices and their dual complex adjoint matrices regarding the eigenvalue problem. Under this framework, the computation of eigenpairs for dual quaternion matrices is reduced to those for dual complex matrices. Building on this idea, \cite{chen2024Applications} proposed a novel and efficient algorithm (Algorithm \ref{Power_method_algorithm}) to compute the eigenvalues of dual quaternion Hermitian matrices, which can be used to effectively address the subproblem \eqref{q2}. This suggests that dual complex adjoint matrices may also serve as a powerful tool to tackle the optimal rank‑$k$ approximations of dual quaternion Hermitian matrices under the $F^*$-norm.  Since the two‑block coordinate descent method lacks global‑convergence guarantees, a robust initialization is essential to accelerate convergence and avoid poor local optima. In \cite{hadi2023mathrm}, an efficient spectral method to SE(3) synchronization via dual quaternions was proposed. Viewing PGO as a synchronization task, we adopt this spectral method to generate an effective initial iterate.

The main contributions of this paper are summarized in the following:
\begin{itemize} 
\item Explicit solutions for rank‑$k$ approximations. We derive explicit solutions for the optimal rank‑$k$ approximation of dual quaternion Hermitian matrices under both the $F$‑norm and the $F^*$‑norm, and we further present efficient algorithms to solve these approximations.
\item Robust reformulation under the $F^*$‑norm. We reformulate the PGO problem using the more appropriate and robust $F^*$‑norm and solve it via the two-block coordinate descent method, resulting in significantly improved experimental accuracy.  
\item Spectral initialization strategy. We adopt an efficient spectral method to generate the initial estimate, which efficiently reduces the number of iterations needed by the two‑block coordinate descent method.
\item  Effective algorithmic refinements. We further improve both the accuracy and efficiency of the algorithm by incorporating proper parameter selection and stagnation-based termination criteria.
\end{itemize}

The paper is organized as follows. In Section \ref{Preliminaries}, we provide an overview of the fundamental definitions and key results for dual quaternions and dual quaternion matrices. Section \ref{section3} introduces the dual complex adjoint matrix along with its useful properties. In Section \ref{section4}, we derive explicit solutions for the optimal rank‑$k$ approximations of dual quaternion Hermitian matrices under both the $F$‑norm and the $F^*$‑norm. In Section \ref{section5}, we improve the two‑block coordinate descent method by reformulating the PGO problem under the $F^*$‑norm and details a suite of algorithmic enhancements—proper parameter selection, stagnation-based termination criteria, and a spectral initialization strategy—to accelerate convergence. Numerical experiments validating the performance of our method are presented in Section \ref{section6}, and concluding remarks appear in Section \ref{section7}.

\section{Preliminary}\label{Preliminaries}

In this section, some notation and definitions are introduced, which will be used in the sequel.

\subsection{Dual complex numbers and dual quaternions}

Denote $\mathbb{R}$ and $\mathbb{C}$ as the sets of real numbers and complex numbers, respectively. Denote the set of dual complex numbers as
\begin{displaymath}
   \mathbb{DC}=\{{a}=a_{st}+a_{\mathcal I}\varepsilon|a_{st},a_{\mathcal I}\in \mathbb{C}\},
  \end{displaymath}
where $\varepsilon$ is the infinitesimal unit, satisfying $\varepsilon\neq 0$ and $\varepsilon^{2}=0$.  We call $a_{st}$ and $a_{\mathcal I}$ the standard part and the dual part of $a$, respectively. The infinitesimal unit $\varepsilon$ is commutative in multiplication with complex numbers and quaternions. We say that $a$ is {\it appreciable} if $a_{st}\neq 0$. The {\it conjugate} \cite{Qi2022} of $a$ is
${a}^{\ast }=a_{st}^{\ast}+a_{\mathcal I}^{\ast }\varepsilon,$ where $a_{st}^{\ast}$ and $a_{\mathcal I}^{\ast}$ are conjugates of $a_{st}$ and $a_{\mathcal I}$, respectively. The {\it absolute value} \cite{Qi2022} of ${a}$ is
 \begin{displaymath}\left |{a}\right |=\begin{cases}
\left | a_{st}\right |+{\rm sgn}( a_{st}) a_{\mathcal I}\varepsilon, & {\rm if} \quad a_{st}\neq 0, \\
\left | a_{\mathcal I} \right |\varepsilon, & \text{\rm otherwise.}
\end{cases}
\end{displaymath}
If $a_{st},a_{\mathcal I}\in \mathbb{R}$, then ${a}$ is called a {\it dual number}. Denote the set of dual numbers as
\begin{displaymath}
   \mathbb{D}=\{{a}=a_{st}+a_{\mathcal I}\varepsilon|a_{st},a_{\mathcal I}\in \mathbb{R}\}.
  \end{displaymath}
The order operation for dual numbers was defined in \cite{cui2024power}. For two dual numbers ${a}=a_{st}+a_{\mathcal I}\varepsilon$ and ${b}=b_{st}+b_{\mathcal I}\varepsilon$, we say that ${a} > {b}$ if
$$\text{$a_{st} > b_{st}$ ~~~{\rm or}~~~ $a_{st}=b_{st}$ {\rm and} $a_{\mathcal I} > b_{\mathcal I}$}.$$
When $b_{st}\ne 0$ or $a_{st}=b_{st}=0$ and $b_{\mathcal I}\ne 0$, we can define the division operation \cite{Qi2022} of dual numbers as
$$\frac{{a}}{{b}}=\begin{cases}
\dfrac{ a_{st}}{b_{st}}+\left (\dfrac{ a_{\mathcal I}}{b_{st}}-\dfrac{ a_{st}b_{\mathcal I}}{b_{st}b_{st}} \right )\varepsilon, & {\rm if} \quad b_{st}\neq 0, \\
\dfrac{ a_{\mathcal I}}{b_{\mathcal I}}+c\varepsilon, & {\rm if} \quad a_{st}=b_{st}=0, ~b_{\mathcal I}\ne 0
\end{cases}
$$
where $c$ is an arbitrary complex number.

For two dual complex numbers ${a}=a_{st}+a_{\mathcal I}\varepsilon, {b}=b_{st}+b_{\mathcal I}\varepsilon\in \mathbb{DC}$, the addition and multiplication of $a$ and $b$ are defined as 
  \begin{displaymath}
  \begin{aligned}
  &{a}+{b}={b}+{a}=\left (a_{st}+b_{st} \right )+\left (a_{\mathcal I}+b_{\mathcal I} \right )\varepsilon,\\
&{a}{b}={b}{a}=a_{st}b_{st} +\left (a_{st}b_{\mathcal I}+a_{\mathcal I}b_{st} \right )\varepsilon.
\end{aligned}
\end{displaymath}

Denote the set of quaternions as
\begin{displaymath}
 \mathbb{Q}=\{\tilde{q}=q_{0}+q_{1} \ii+q_{2} \jj+q_{3} \kk|~q_{0},q_{1},q_{2},q_{3}\in\mathbb{R}\},
  \end{displaymath}
 where $\ii, \jj, \kk$ are three imaginary units of quaternions, satisfying
 \begin{displaymath}
  \begin{aligned}
  &\ii\jj=-\jj\ii=\kk,\quad \jj\kk=-\kk\jj=\ii,\quad \kk\ii=-\ii\kk=\jj,\\
  &\ii^2=\jj^2= \kk^2=-1.
 \end{aligned}
\end{displaymath}
A quaternion $\tilde{q}\in\mathbb{Q}$ can also be represented as $\tilde{q}=\left [q_{0}, q_{1},q_{2},q_{3}\right ]$, which is a real four-dimensional vector. We can rewrite $\tilde{q}=[ q_{0},\vec{q} ]$, where $\vec{q}$ is a real three-dimensional vector \cite{Cheng2016,Daniilidis1999,Wei2013}. The conjugate of $\tilde{p}$ is $\tilde{p}^{\ast}=\left [p_{0}, -\vec{p}\right ]$ and the magnitude of $\tilde{p}$ is
$\left |\tilde{p}\right|=\sqrt{p_{0}^{2}+\left \| \vec{p} \right \|_2^2}$.

The zero element in $\mathbb{Q}$ is $\tilde{0}=\left [0,0,0,0\right ]$ and the unit element is $\tilde{1}=\left [1,0,0,0\right ]$. For any two quaternions $\tilde{p}=\left [p_{0}, \vec{p}\right ]$ and $\tilde{q}=\left [q_{0}, \vec{q}\right ]$, the addition and multiplicity of $\tilde{q}$ and $\tilde{p}$ are defined as
 \begin{displaymath}
\tilde{p}+\tilde{q}=\tilde{q}+\tilde{p}=\left [p_{0}+q_{0},\vec{p}+\vec{q}\right ],\quad
\tilde{p}\tilde{q}=\left [p_{0}q_{0}-\vec{p}\cdot \vec{q}, p_{0}\vec{q}+q_{0}\vec{p}+\vec{p}\times \vec{q}\right ].
\end{displaymath}
The multiplication of quaternions satisfies the distributive law but is noncommutative. If $\tilde{q}\tilde{p}=\tilde{p}\tilde{q}=\tilde{1}$, then the quaternion $\tilde{p}$ is said to be invertible and its inverse is $\tilde{p}^{-1}=\tilde{q}$.

Denote the set of unit quaternions as $\mathbb{U}=\{\tilde{p}\in \mathbb{Q}|~|\tilde{p} |=1\}$. For any two unit quaternions $\tilde{p},\tilde{q}\in \mathbb{U}$, we have $\tilde{p}\tilde{q}\in \mathbb{U}$ and $\tilde{p}^{\ast }\tilde{p}=\tilde{p}\tilde{p}^{\ast }=\tilde{1}$, i.e., $\tilde{p}$ is invertible and $\tilde{p}^{-1}=\tilde{p}^{\ast }$. Generally, if a quaternion $\tilde{p}\ne \tilde{0}$, then we have  $\frac{\tilde{p}^{\ast }}{|\tilde{p}|^2}\tilde{p}=\tilde{p}\frac{\tilde{p}^{\ast }}{|\tilde{p}|^2}=\tilde{1}$, and hence $\tilde{p}^{-1}=\frac{\tilde{p}^{\ast }}{|\tilde{p}|^2}$.

Denote the set of dual quaternions as
 \begin{displaymath}
 \hat{\mathbb{Q}}=\{\hat{p}=\tilde{p}_{st}+\tilde{p}_{\mathcal I}\varepsilon|~\tilde{p}_{st}, \tilde{p}_{\mathcal I} \in \mathbb{Q}\}.
\end{displaymath}
We call $\tilde{p}_{st}$ and $\tilde{p}_{\mathcal I}$ the standard part and dual part of $\hat{p}$, respectively. If $\tilde{p}_{st}\neq \tilde{0}$, then we say that $\hat{p}$ is {\it appreciable}.
The conjugate of $\hat{p}$ is
$\hat{p}^{\ast }=\tilde{p}_{st}^{\ast}+\tilde{p}_{\mathcal I}^{\ast }\varepsilon$ and the magnitude \cite{Qi2022} of $\hat{p}$ is
\begin{equation*}
\left | \hat{p}\right |=\begin{cases}
\left | \tilde{p}_{st}\right |+\frac{{\rm sc}(\tilde{p}_{st}^{\ast }\tilde{p}_{\mathcal I})}{\left |\tilde{p}_{st} \right |}\varepsilon, & \text{if} \quad \tilde{p}_{st}\neq \tilde{0}, \\
\left | \tilde{p}_{\mathcal I} \right |\varepsilon, & \text{if} \quad \tilde{p}_{st}= \tilde{0},
\end{cases}
\end{equation*}
where ${\rm sc}(\tilde{p})=\frac{1}{2}(\tilde{p}+\tilde{p}^{\ast})$.
We say that $\hat{p}$ is a {\it unit dual quaternion} if $\left |\hat{p} \right |=1$.

The zero element in $\hat{\mathbb{Q}}$ is $\hat{0}=\tilde{0}+\tilde{0}\varepsilon$ and the unit element is $\hat{1}=\tilde{1}+\tilde{0}\varepsilon$. For two dual quaternions $\hat{p}=\tilde{p}_{st}+\tilde{p}_{\mathcal I}\varepsilon, \hat{q}=\tilde{q}_{st}+\tilde{q}_{\mathcal I}\varepsilon$, addition and multiplicity of $\hat{p}$ and $\hat{q}$ are
\begin{equation*}
\hat{p}+\hat{q}=\hat{q}+\hat{p}=\left (\tilde{p}_{st}+\tilde{q}_{st} \right )+\left (\tilde{p}_{\mathcal I}+\tilde{q}_{\mathcal I} \right )\varepsilon,\qquad
\hat{p}\hat{q}=\tilde{p}_{st}\tilde{q}_{st} +\left (\tilde{p}_{st}\tilde{q}_{\mathcal I}+\tilde{p}_{\mathcal I}\tilde{q}_{st} \right )\varepsilon.
\end{equation*}
If $\tilde{q}_{st}\ne \tilde{0}$ or  $\tilde{p}_{st}=\tilde{q}_{st}=\tilde{0}$ and $\tilde{q}_{\mathcal I}\ne \tilde{0}$, the division \cite{cui2024power} of $\hat{p}$ and $\hat{q}$ is defined as
\begin{equation*}
\frac{ \tilde{p}_{st}+ \tilde{p}_{\mathcal I}\varepsilon}{\tilde{q}_{st}+ \tilde{q}_{\mathcal I}\varepsilon}=\begin{cases}
\dfrac{ \tilde{p}_{st}}{\tilde{q}_{st}}+\left (\dfrac{ \tilde{p}_{\mathcal I}}{\tilde{q}_{st}}-\dfrac{ \tilde{p}_{st}\tilde{q}_{\mathcal I}}{\tilde{q}_{st}\tilde{q}_{st}} \right )\varepsilon, & \text{if} \quad \tilde{q}_{st}\neq \tilde{0}, \\
\dfrac{ \tilde{p}_{\mathcal I}}{\tilde{q}_{\mathcal I}}+\tilde{c}\varepsilon, & \text{if} \quad \tilde{p}_{st}=\tilde{q}_{st}=\tilde{0},~\tilde{q}_{\mathcal I}\ne \tilde{0}
\end{cases}
\end{equation*}
where $\tilde{c}\in \mathbb{Q}$ is an arbitrary quaternion.  If $\hat{q}\hat{p}=\hat{p}\hat{q}=\hat{1}$, then $\hat{p}$ is called invertible and the inverse of $\hat{p}$ is $\hat{p}^{-1}=\hat{q}$.

Denote the set of unit dual quaternions as $\hat{\mathbb{U}}=\{\hat{p}\in\hat{\mathbb{Q}}|~|\hat{p}|=1\}$. Let $\hat{p}\in \hat{\mathbb{U}}$, then $\hat{p}^{\ast }\hat{p}=\hat{p}\hat{p}^{\ast }=\hat{1}$. Hence, $\hat{p}$ is invertible and $\hat{p}^{-1}=\hat{p}^{\ast}$. Generally, if $\hat{p}=\tilde{p}_{st}+\tilde{p}_{\mathcal I}\varepsilon\in\hat{\mathbb{Q}}$ is appreciable, then $\hat{p}^{-1}=\tilde{p}^{-1}_{st}-\tilde{p}^{-1}_{st}\tilde{p}_{\mathcal I}\tilde{p}^{-1}_{st}\varepsilon$.

A unit dual quaternion $\hat{u}$ is called the projection of $\hat{q}\in \hat{\mathbb Q}$ onto the set $\hat{\mathbb{U}}$, if 
\begin{equation*}
\hat{u}\in \underset{\hat{v}\in\hat{\mathbb{U}}}{\arg\min} ~
|\hat{v}-\hat{q}|^2.
\end{equation*}
A feasible solution was given in \cite{cui2024power}. Denote $\hat{q}=\tilde{q}_{st}+\tilde{q}_{\mathcal I}\varepsilon$. If $\tilde{q}_{st}\neq \tilde{0}$, let 
\begin{equation*}
\hat{u}=\frac{\hat{q}}{\left | \hat{q}\right |}=\frac{\tilde{q}_{st}}{\left | \tilde{q}_{st}\right |}+\left ( \frac{\tilde{q}_{\mathcal I}}{\left | \tilde{q}_{st}\right |}-\frac{\tilde{q}_{st}}{\left | \tilde{q}_{st}\right |}sc\left( \frac{\tilde{q}_{st}^{\ast }}{\left | \tilde{q}_{st}\right |}\frac{\tilde{q}_{\mathcal I}}{\left | \tilde{q}_{st}\right |}\right)\right )\varepsilon.
\end{equation*}
If $\tilde{q}_{st}=\tilde{0}$ and $\tilde{q}_{\mathcal I}\neq \tilde{0}$, let
$\hat{u}=\frac{\tilde{q}_{\mathcal{I}}}{\left | \tilde{q}_{\mathcal{I}}\right |}+\tilde{c}\varepsilon \in \hat{\mathbb{U}},$
where $\tilde{c}$ is an arbitrary quaternion, satisfying $sc(\tilde{q}_{\mathcal{I}}^{\ast }\tilde{c})=\tilde{0}$.
Then $\hat{u}$ is the projection of $\hat{q}$ onto the set $\hat{\mathbb{U}}$.

\subsection{Dual quaternion matrix}

The sets of $n\times m$-dimensional dual number matrices, dual complex matrices, quaternion matrices, and dual quaternion matrices are denoted as $\mathbb{D}^{n\times m}$,
$\mathbb{DC}^{n\times m}$, $\mathbb{Q}^{n\times m}$, and $\hat{\mathbb{Q}}^{n\times m}$, respectively.
Their zero elements are ${O}^{n\times m}$, $\hat{O}^{n\times m}$, $\tilde{\mathbf{O}}^{n\times m}$ and  $\hat{\mathbf{O}}^{n\times m}$.
Denote $\hat{I}_n$, $\tilde{\mathbf{I}}_n$ and $\hat{\mathbf{I}}_n$ as the unit element of $\mathbb{DC}^{n\times n}$, $\mathbb{Q}^{n\times n}$, and $\hat{\mathbb{Q}}^{n\times n}$, respectively. If every element of a dual quaternion matrix is a unit dual quaternion, then we refer to it as 
a unit dual quaternion matrix. The set of $n\times m$-dimensional unit dual quaternion matrices is denoted as $\hat{\mathbb{U}}^{n\times m}$.

A dual complex matrix $\hat{P}=P_{st}+P_{\mathcal I}\varepsilon \in \mathbb{DC}^{n\times m}$ has standard part $P_{st}\in \mathbb C^{n\times m}$ and dual part $P_{\mathcal I}\in \mathbb C^{n\times m}$. The transpose and the conjugate of $\hat{P}=(\hat{p}_{ij})$ are $\hat{P}^{T}=(\hat{p}_{ji})$ and $\hat{P}^{\ast}=(\hat{p}_{ji}^{\ast })$, respectively. If $\hat{P}^{\ast}=\hat{P}$, then $\hat{P}$ is a dual complex Hermitian matrix. A dual complex matrix $\hat{U}\in \mathbb{DC}^{n\times n}$ is a unitary matrix, if $\hat{U}^\ast\hat{U}=\hat{U}\hat{U}^\ast=\hat{I}_n$.

The 2-norm of dual complex vector $\hat{\mathbf{\mathbf x}}=\mathbf x_{st}+\mathbf x_{\mathcal I}\varepsilon=(\hat{x}_{i})\in \mathbb{DC}^{n\times 1}$ is defined as
\begin{equation*}
\left \| \hat{\mathbf x }\right \|_{2}=\begin{cases}
\sqrt{\hat{\mathbf x }^*\hat{\mathbf x }}, & \text{if} \quad \mathbf x_{st}\neq O, \\
\|\mathbf{x}_{\mathcal{I}}\|_2\varepsilon, & \text{if} \quad \mathbf x_{st}=O.
\end{cases}  
\end{equation*}

Suppose that $\hat{P}=P_1+P_2\varepsilon \in \mathbb{DC}^{n\times m}$,
the $F$-norm and the $F^*$-norm of a dual complex matrix $\hat{P}$ are defined by
\begin{equation*}
\| \hat{P}\|_{F}=\begin{cases}
\left \| P_1\right \|_{F}+\frac{sc(tr (P_1^{\ast }P_2))}{\left \| P_1 \right \|_{F}}\varepsilon, & \text{if} \quad P_1\neq O, \\
\left \| P_2\right \|_{F}\varepsilon, & \text{if} \quad P_1=O,
\end{cases} 
\end{equation*}
where $tr(P)$ is the sum of the diagonal elements of $P$, and
\begin{equation*}
\| \hat{P}\|_{F^*}=\begin{cases}
\| P_1\|_{F}+\frac{\| P_2\|^2_{F}}{2\| P_1\|_{F}}\varepsilon, & \text{if} \quad P_1\neq O, \\
\| P_2\|_{F}\varepsilon, & \text{if} \quad P_1=O.
\end{cases} 
\end{equation*}

The definition of the rank-k decomposition of a dual complex matrix is defined in \cite{wang2023dual}.
\begin{definition} \label{de2.13}
Let $\hat{A}=A_1+A_2\varepsilon\in \mathbb{DC}^{m\times n}$, $\hat{B}=B_1+B_2\varepsilon\in \mathbb{DC}^{m\times k}$ and $\hat{C}=C_1+C_2\varepsilon\in \mathbb{DC}^{n\times k}$. If the complex matrices $B_1$ and $C_1$ are full rank, and $\hat{A}=\hat{B}\hat{C}^*$, then $\hat{B}\hat{C}^*$ is called the rank-k decomposition of $\hat{A}$.
If a dual complex matrix $\hat{A}$ has a rank-k decomposition, we say that $\hat{A}$ is rank-k.
\end{definition}

The following two lemmas give the optimal rank-k approximation of dual complex matrices under the $F$-norm and $F^*$-norm \cite{wei2024singular}.

\begin{lemma}\label{lemma2F}
Let $\hat{A}=A_1+A_2\varepsilon\in \mathbb{DC}^{m\times n}$, $\hat{U}=U_1+U_2\varepsilon\in \mathbb{DC}^{m\times k}$, $\hat{V}=V_1+V_2\varepsilon\in \mathbb{DC}^{m\times k}$, $\hat{\Sigma}= \Sigma_1+\Sigma_2\varepsilon\in \mathbb{DC}^{k\times k}$. If
$U_1\Sigma_1V_1^*$ is the best rank-k approximation of $A_1$ under the Frobnius norm, then $\hat{U},\hat{V},\hat{\Sigma}$ is the optimal solution to the optimization problem
\begin{align*}\min_{\hat{U},\hat{V},\hat{\Sigma}}
&~ \| A-\hat{U}\hat{\Sigma}\hat{V}^* \|_F^2,
\\\text {s.t.}
~&\hat{U}^*\hat{U}=\hat{V}^*\hat{V}=\hat{I}_k, (\hat{\Sigma})_{ij}=\hat{0}~(\forall i\neq j),(\hat{\Sigma})_{ii}>\hat{0}~(\forall i=1,\cdots,k).
\end{align*}
\end{lemma}

\begin{lemma}\label{lemma2F*}
Let $\hat{A}=A_1+A_2\varepsilon\in \mathbb{DC}^{m\times n}$, $\hat{U}=U_1+U_2\varepsilon\in \mathbb{DC}^{m\times k}$, $\hat{V}=V_1+V_2\varepsilon\in \mathbb{DC}^{m\times k}$, $\hat{\Sigma}= \Sigma_1+\Sigma_2\varepsilon\in \mathbb{DC}^{k\times k}$. Let $U_1\Sigma_1V_1^*$ be the best rank-k approximation of $A_1$ under the Frobnius norm. Suppose that $\Sigma_1=\mathrm{diag}(\sigma_1 I_{r_1},\sigma_2I_{r_2},\cdots,\sigma_\iota I_{r_\iota})$, satisfies $\sigma_1>\sigma_2>\cdots>\sigma_\iota$. Let
\begin{align*}
&\quad U_2=U_1\left(\mathrm{sym}(R\Sigma_1)\odot\Delta+\Omega+\Psi\right)+(I_m-U_1U_1^*)A_2V_1\Sigma_1^{-1},
\\&\quad V_2=V_1\left (\mathrm{sym}(\Sigma_1R)\odot\Delta+ \Omega\right )+(I_n-V_1V_1^*)A_2^*U_1\Sigma_1^{-1},
\\&\quad \Sigma_2=\frac12\mathrm{diag}(\mathrm{sym}(R)),\end{align*}
where $R=U_1^*A_2V_1$, denote $R=(R_{ij})$, $R_{ij}\in \mathbb{C}^{r_i\times r_i}$, $\odot$ is the Hadamard product of matrices, $\mathrm{sym}(A)=A+A^*$, denote $\Omega=\mathrm{diag}(\Omega_{11},
\Omega_{22},\cdots,\Omega_{\iota \iota })$, 
$\Omega_{ii}\in\mathbb{C}^{r_i\times r_i}$, 
$\Omega_{ii}=-\Omega_{ii}^*,i=1,2,\cdots,\iota$, 
denote $\Delta=(\Delta_{ij})$, $\Delta_{ij}\in \mathbb{C}^{r_i\times r_i}$, $\Delta_{ii}=O^{r_i\times r_i}$, $\Delta_{ij}=\frac{1_{r_i\times r_j}}{\sigma_j^2-\sigma_i^2}$, for $i\ne j$,
denote $\Psi=(\Psi_{ij})$, $\Psi_{ij}\in \mathbb{C}^{r_i\times r_i}$, $\Psi_{ii}=\frac{R_{ii}-R_{ii}^*}{2\sigma_1}$, $\Psi_{ij}=O^{r_i\times r_j}$, for $i\ne j$,
then $\hat{U},\hat{V},\hat{\Sigma}$ is the optimal solution to the optimization problem
\begin{align*}\min_{\hat{U},\hat{V},\hat{\Sigma}}
&~ \| A-\hat{U}\hat{\Sigma}\hat{V}^*\|_{F^*}^2,
\\\text {s.t.} 
~&\hat{U}^*\hat{U}=\hat{V}^*\hat{V}=\hat{I}_k, (\hat{\Sigma})_{ij}=\hat{0}~(\forall i\neq j),(\hat{\Sigma})_{ii}>\hat{0}~(\forall i=1,\cdots,k).
\end{align*}
Specifically, $(\hat{U}\hat{\Sigma}\hat{V}^*)_{\mathcal{I}}=A_2-(I_m-U_1U_1^*)A_2(I_n-V_1V_1^*).$
\end{lemma}

A quaternion matrix $\tilde{\mathbf Q}\in \mathbb Q^{m\times n}$ can be
expressed as $\tilde{\mathbf Q}= Q_1+ Q_2 \ii+ Q_3
\jj+ Q_4 \kk$ with $ Q_1, Q_2, Q_3, Q_4\in \mathbb{R}^{m\times n}$.
Let $P_1= Q_1+Q_2 \ii$ and $P_2= Q_3+ Q_4 \ii$, then $\tilde{\mathbf Q}$ can be rewritten as
$\tilde{\mathbf Q}= P_1+ P_2 \jj$.
The $F$-norm of a quaternion matrix $\tilde{\mathbf{Q}}=(\tilde{q}_{ij}) \in \tilde{\mathbb{Q}}^{m\times n}$ is defined as
$\| \tilde{\mathbf{Q}}\| _F=\sqrt{\sum_{ij}^{} |\tilde{q}_{ij}|^2}$.
The magnitude of a quaternion vector $\tilde{\mathbf{x}}=(\tilde{x}_{i}) \in \tilde{\mathbb{Q}}^{n\times 1}$ is defined as
$\left \| \tilde{\mathbf{x}} \right \|=\sqrt{\sum_{i=1}^{n}|\tilde{x}_{i}|^2}.$

A dual quaternion matrix $\hat{\mathbf Q}=\tilde{\mathbf Q}_{st}+\tilde{\mathbf Q}_{\mathcal I}\varepsilon \in \hat{\mathbb{Q}}^{m\times n}$ has standard part  $\tilde{\mathbf Q}_{st}\in \mathbb Q^{m\times n}$ and dual part $\tilde{\mathbf Q}_{\mathcal I}\in \mathbb Q^{m\times n}$. If $\tilde{\mathbf Q}_{st}\neq \tilde{\mathbf O}$, then $\hat{\mathbf Q}$ is called appreciable. The transpose and conjugate of $\hat{\mathbf{Q}}=(\hat{q}_{ij})$ are
$\hat{\mathbf{Q}}^{T}=(\hat{q}_{ji})$ and $\hat{\mathbf{Q}}^{\ast}=(\hat{q}_{ji}^{\ast })$, respectively. 
Denote the set of unitary dual quaternion matrix with dimension $n$ as
$$
\hat{\mathbb{U}}^n_2=\{\hat{\mathbf{U}}\in \hat{\mathbb{Q}}^{n\times n}|~\hat{\mathbf{U}}^\ast\hat{\mathbf{U}}=\hat{\mathbf{U}}\hat{\mathbf{U}}^\ast=\hat{\mathbf{I}}_n\}.
$$
Denote the set of dual quaternion Hermitian matrix with dimension $n$ as
$$
\hat{\mathbb{H}}^n=\{\hat{\mathbf{Q}}\in \hat{\mathbb{Q}}^{n\times n} |~\hat{\mathbf{Q}}^{\ast}=\hat{\mathbf{Q}}\}.
$$

The $2$-norm of a dual quaternion vector  $\hat{\mathbf{x}}=\tilde{\mathbf{x}}_{st}+\tilde{\mathbf{x}}_{\mathcal{I}}\varepsilon \in \hat{\mathbb{Q}}^{n\times 1}$ is defined as
\begin{equation*}
\| \hat{\mathbf x}\|_{2}=\begin{cases}
\sqrt{\hat{\mathbf x}^*\hat{\mathbf x}}, & \text{if} \quad \tilde{\mathbf x }_{st}\neq \tilde{\mathbf O }, \\
\|\tilde{\mathbf{x}}_{\mathcal{I}}\|\varepsilon, & \text{if} \quad \tilde{\mathbf x }_{st}=\tilde{\mathbf O },
\end{cases}
\end{equation*}
and denote $\hat{\mathbb{Q}}_2^{n\times 1}=\{\hat{\mathbf{x}}\in \hat{\mathbb{Q}}^{n\times 1}|~\| \hat{\mathbf x}\|_{2}=1\}$.
The $F^{R}$-norm of a dual quaternion matrix $\hat{\mathbf Q}=\tilde{\mathbf Q}_{st}+\tilde{\mathbf Q}_{\mathcal I}\varepsilon$ is defined by
\begin{equation*}
 \| \hat{\mathbf Q} \|_{F^{R}}=\sqrt{ \| \tilde{\mathbf Q}_{st} \|_{F}^{2}+\| \tilde{\mathbf Q}_{\mathcal I}\|_{F}^{2}}.
\end{equation*}
For details of these definitions, see \cite{cui2024power,Qi2022}.

A unit quaternion vector $\hat{\mathbf u}$ is called the projection of a dual quaternion vector $\hat{\mathbf x}$ onto the set of dual quaternion vector with unit $2$-norm, i.e.,$\hat{\mathbb{Q}}_2^{n\times 1}$, if  
\begin{equation}\label{def11}
\hat{\mathbf{u}}\in\underset{\hat{\mathbf{v}}\in\hat{\mathbb{Q}}_2^{n\times1}}{\operatorname*{\arg\min}}~\|\hat{\mathbf{v}}-\hat{\mathbf{x}}\|_2^2.
\end{equation}
A feasible solution was given in \cite{cui2024power}.
If $\tilde{\mathbf{x}}_{st}\neq \tilde{\mathbf{O}}^{n\times 1}$, denote $$\hat{\mathbf u}=\frac{\hat{\mathbf x}}{\left \| \hat{\mathbf x}\right \|_2}=\frac{\tilde{\mathbf x}_{st}}{\left \| \tilde{\mathbf x}_{st}\right \|}+\left ( \frac{\tilde{\mathbf x}_{\mathcal I}}{\left \| \tilde{\mathbf x}_{st}\right \|}-\frac{\tilde{\mathbf x}_{st}}{\left \| \tilde{\mathbf x}_{st}\right \|}sc\left( \frac{\tilde{\mathbf x}_{st}^{\ast }}{\left \| \tilde{\mathbf x}_{st}\right \|}\frac{\tilde{\mathbf x}_{\mathcal I}}{\left \| \tilde{\mathbf x}_{st}\right \|}\right) \right )\varepsilon.
$$ If $\tilde{\mathbf x}_{st}= \tilde{\mathbf{O}}^{n\times 1}$ and $\tilde{\mathbf x}_{\mathcal I}\neq \tilde{\mathbf{O}}^{n\times 1}$, denote 
$$\hat{\mathbf u}=\frac{\tilde{\mathbf x}_{\mathcal{I}}}{\left \| \tilde{\mathbf x}_{\mathcal{I}}\right \|}+\tilde{\mathbf c}\varepsilon \in \hat{\mathbb{Q}}_2^{n\times1},$$ where $\tilde{\mathbf c}$ is arbitrary queternion vector satisfying
$sc(\tilde{\mathbf x}_{\mathcal{I}}^{\ast }\tilde{\mathbf c})=\tilde{0}$.
Then $\hat{\mathbf u}$ is the projection of $\hat{\mathbf x}$ onto the set 
$\hat{\mathbb{Q}}_2^{n\times 1}$.

The following definition of eigenvalue and eigenvector of dual quaternion matrices was given in \cite{Li2023}.
\begin{definition}\label{def:righteig}
Let $\hat{\mathbf Q}\in \hat{\mathbb{Q}}^{n\times n}$.
If there exist $\hat{\lambda} \in \hat{\mathbb{Q}}$ and $\hat{\mathbf x}\in \hat{\mathbb{Q}}^{n\times 1}$ such that $\hat{\mathbf x}$ is appreciable and
$\hat{\mathbf Q}\hat{\mathbf x}=\hat{\mathbf x}\hat{\lambda},$ then $\hat{\lambda}$ is called a {\bf right eigenvalue} of $\hat{\mathbf Q}$ with $\hat{\mathbf x}$ as the corresponding {\bf right eigenvector}.
If there exist $\hat{\lambda} \in \hat{\mathbb{Q}}$ and $\hat{\mathbf x}\in \hat{\mathbb{Q}}^{n\times 1}$, where $\hat{\mathbf x}$ is appreciable, such that
$\hat{\mathbf Q}\hat{\mathbf x}=\hat{\lambda}\hat{\mathbf x},$ then we call $\hat{\lambda}$ is a {\bf left eigenvalue} of $\hat{\mathbf Q}$ with $\hat{\mathbf x}$ as an associated {\bf left eigenvector}.
\end{definition}

Let $\hat{\lambda}$ be a right eigenvalue of $\hat{\mathbf Q}$ with $\hat{\mathbf x}$ as the corresponding right eigenvector. Since a dual number is commutative with a dual quaternion vector, then if $\hat{\lambda}$ is a dual number, it is also a left eigenvalue of $\hat{\mathbf{Q}}$. In this case, we simply call $\hat{\lambda}$ an {\bf eigenvalue} of $\hat{\mathbf{Q}}$ with $\hat{\mathbf{x}}$ as an associated {\bf eigenvector}.

\section{Dual Complex Adjoint Matrix}\label{section3}
In this section, we introduce the dual complex adjoint matrix with some useful properties, and a refined power method for computing the eigenpairs of dual quaternion Hermitian matrices.

First, we introduce the {\it dual complex adjoint matrix} of a dual quaternion matrix $\hat{\mathbf Q}=(A_1+A_2 \jj)+( A_3+ A_4 \jj)\varepsilon$ as introduced in \cite{chen2024dual}, via the mapping
\begin{equation*}\label{JJ}
\mathcal{J}(\hat{\mathbf Q})=\begin{bmatrix}
 A_1 & A_2 \\
-\overline{A_2}   & \overline{ A_1}
\end{bmatrix}
+\begin{bmatrix}
 A_3 &  A_4 \\
-\overline{ A_4}   & \overline{A_3}
\end{bmatrix} \varepsilon.
\end{equation*}
Here $\mathcal{J}$ defines a bijection from the set  $\mathcal{J}:\hat{\mathbb{Q}}^{m\times n}\to DM(\mathbb{C}^{m\times n})$, where
\begin{equation*}
DM(\mathbb{C}^{m\times n})=
\left.\left\{\mathcal{J}(\hat{\mathbf Q})=\begin{bmatrix}
A_1 &  A_2 \\
-\overline{ A_2}   & \overline{ A_1}
\end{bmatrix}
+\begin{bmatrix}
 A_3 & A_4 \\
-\overline{A_4}   & \overline{ A_3}
\end{bmatrix} \varepsilon \right| A_1, A_2, A_3, A_4\in\mathbb{C}^{m\times n}\right\}
\end{equation*}
is the collection of all dual complex adjoint matrices of size $2m\times2n$.  In particular, $DM(\mathbb{C}^{m \times n})$ is a subset of $\mathbb {DC}^{2m \times 2n}$.

The following lemma lists some useful properties about the bijection $\mathcal{J}$, see \cite{chen2024dual}.

\begin{lemma}\label{lemma3.2}
Let $\hat{\mathbf P},\hat{\mathbf P}_1\in \hat{\mathbb{Q}}^{m\times k}$, $\hat{\mathbf Q} \in \hat{\mathbb{Q}}^{k\times n}$, $\hat{\mathbf R}\in \hat{\mathbb{Q}}^{n\times n}$, then
\begin{itemize}
\item[{\rm (i)}]  $\mathcal{J}(\hat{\mathbf O}^{m\times n})=\hat{O}^{2m\times 2n},\mathcal{J}(\hat{\mathbf I}_n)=\hat{I}_{2n}$.
\item[{\rm (ii)}]  $\mathcal{J}(\hat{\mathbf P}\hat{\mathbf Q})=\mathcal{J}(\hat{\mathbf P})\mathcal{J}(\hat{\mathbf Q})$.
\item[{\rm (iii)}]  $\mathcal{J}(\hat{\mathbf P}+\hat{\mathbf P}_1)=\mathcal{J}(\hat{\mathbf P})+\mathcal{J}(\hat{\mathbf P}_1)$.
\item[{\rm (iv)}]  $\mathcal{J}(\hat{\mathbf P}^*)=\mathcal{J}(\hat{\mathbf P})^*$.
\item[{\rm (v)}]  $\mathcal{J}(\hat{\mathbf R})$ is unitary (Hermitian) if and only if $\hat{\mathbf R}$ is unitary (Hermitian).
\item[{\rm (vi)}]  $\mathcal{J}$ is an isomorphism from ring $(\hat{\mathbb{Q}}^{n\times n},+,\cdot)$ to ring $(DM(\mathbb{C}^{n\times n}),+,\cdot)$.
\end{itemize}
\end{lemma}

For convenience, we introduce the following notation. Define the bijection $\mathcal{F}:\hat{\mathbb{Q}}^{n\times 1}\to \mathbb{DC}^{2n\times 1}$ by
\begin{equation*}
\mathcal{F}(\mathbf v_{1}+\mathbf v_{2} \jj+(\mathbf v_{3}+\mathbf v_{4} \jj))= 
\begin{bmatrix}
    \mathbf v_{1}\\
    -\overline{\mathbf v_{2}} 
    \end{bmatrix}+
    \begin{bmatrix}
    \mathbf v_{3}\\
    -\overline{\mathbf v_{4}} 
    \end{bmatrix}\varepsilon,  
\end{equation*}
where each $\mathbf v_{i}\in \mathbb C^{n\times 1}$. Its inverse is given by
\begin{equation*}
\mathcal{F}^{-1}\left ( \begin{bmatrix}
    \mathbf u_{1}\\
    \mathbf u_{2}
    \end{bmatrix}+
    \begin{bmatrix}
    \mathbf u_{3}\\
    \mathbf u_{4}
    \end{bmatrix}\varepsilon \right )=
    \mathbf u_{1}-\overline{\mathbf u_{2}} \jj+
(\mathbf u_{3}-\overline{\mathbf u_{4}} \jj  ) \varepsilon,
\end{equation*}
where each $\mathbf u_{i}\in \mathbb C^{n\times 1}$.

The next lemma describes how, via $\mathcal{F}$ and $\mathcal{J}$, the eigenvalue problem for a dual quaternion matrix reduces to one for its dual complex adjoint matrix \cite{chen2024Applications}.

\begin{lemma}\label{corollary3.1}
Let $\hat{\mathbf Q}\in\hat{\mathbb{H}}^n$, $\{ \hat{\lambda}_{i}\}^{n}_{i=1}$ and $\{\hat{\mathbf v}^{i}\}^{n}_{i=1}$ be eigenvalues and corresponding eigenvectors of $\hat{\mathbf Q}$, then
$\hat{\lambda}_1,\hat{\lambda}_1,\hat{\lambda}_2,\hat{\lambda}_2,\cdots,\hat{\lambda}_n,\hat{\lambda}_n$ are all the eigenvalues of $\hat{P}=\mathcal{J}(\hat{\mathbf Q})$, in addition, $\mathcal{F}(\mathbf v^{i})$ and $\mathcal{F}(\mathbf v^{i} j)$ are two orthogonal eigenvectors of $\hat{P}$ with respect to eigenvalue $\hat{\lambda}_{i}$.
On the other hand, if $\hat{\mathbf u}^i\in \mathbb{DC}^{2n\times 1}$ is an
eigenvector of $\hat{P}$ with respect to eigenvalue $\hat{\lambda}_{i}$, then $\mathcal{F}^{-1}(\hat{\mathbf u}^i)$ is an eigenvector of $\hat{\mathbf Q}$ with respect to eigenvalue $\hat{\lambda}_{i}$.
\end{lemma}

Based on these properties, an efficient dual complex adjoint matrix based power method (\ref{Power_method_algorithm}) is introduced in \cite{chen2024Applications}.

\begin{algorithm}
    \caption{DCAM-PM} \label{Power_method_algorithm}
    \begin{algorithmic}
    \REQUIRE $\hat{\mathbf Q}\in \hat{\mathbb H}^n$, initial iteration vector $\hat{\mathbf v}^{(0)}\in \hat{\mathbb Q}^{n\times 1}_2$, the maximal iteration number $k_{max}$ and the tolerance $\delta$. 
    \ENSURE $\hat{\mathbf v}^*$ and $\hat{\lambda}^{( k ) }$.
    \STATE Compute $\hat{P}=\mathcal{J}(\hat{\mathbf Q})$ and $\hat{\mathbf u}^{(0)}=\mathcal{F}(\hat{\mathbf v}^{(0)})$.
    \FOR {$k = 1$ to $k_{max}$}   
    \STATE Update $\hat{\mathbf y}^{(k)}=\hat{P}\hat{\mathbf u}^{(k-1)}$, $\hat{\lambda}^{(k)}=(\hat{\mathbf u}^{(k-1)})^*\hat{\mathbf y}^{(k)}$, $\hat{\mathbf u}^{(k)}=\frac{\hat{\mathbf y}^{( k)}}{\| \hat{\mathbf y}^{(k) } \|_2 }$.
    \STATE If $\| \hat{\mathbf y}^{(k) } -\hat{\mathbf u}^{(k-1) }\hat{\lambda}^{( k ) } \|_{2^R}\le \delta $, then Stop.
    \ENDFOR
    \STATE Compute $\hat{\mathbf v}^*=\mathcal{F}^{-1}(\hat{\mathbf u}^{(k)})$.
    \end{algorithmic}
\end{algorithm}

\section{The optimal rank-k approximation of dual quaternion Hermitian matrices}\label{section4}
In this section, leveraging the properties of the dual complex adjoint matrix, we derive explicit solutions for the optimal rank‑k approximation of a dual quaternion Hermitian matrix under both the $F$-norm and $F^*$-norm.

The definition of the rank of a quaternion matrix is defined as follows \cite{Zhang1997}.
\begin{definition}\label{de3.6}
Let $\tilde{\mathbf{A}}=A_1+A_2 \jj\in \mathbb{Q}^{m\times n}$, the rank of $\tilde{\mathbf{A}}$ is $k$ if and only if the rank of the complex adjoint matrix $J(\tilde{\mathbf{A}})$ is $2k$, where 
$J(\tilde{\mathbf A})=\begin{bmatrix}
 A_1 & A_2 \\
-\overline{A_2}   & \overline{ A_1}
\end{bmatrix}.$
\end{definition}

Similar to the definition of the rank-k decomposition of dual quaternion matrix (Definition \ref{de2.13}), we define the rank-k decomposition of a dual quaternion matrix as follows.
\begin{definition}
Let $\hat{\mathbf{A}}=\tilde{\mathbf A}_1+\tilde{\mathbf A}_2\varepsilon\in \hat{\mathbb{Q}}^{m\times n}$, $\hat{\mathbf{B}}=\tilde{\mathbf B}_1+\tilde{\mathbf B}_2\varepsilon\in \hat{\mathbb{Q}}^{m\times k}$, $\hat{\mathbf{C}}=\tilde{\mathbf C}_1+\tilde{\mathbf C}_2\varepsilon\in \hat{\mathbb{Q}}^{n\times k}$, where $m\ge k$ and $n\ge k$. If the quaternion matrices $\tilde{\mathbf B}_1$ and $\tilde{\mathbf C}_1$ are both full rank and $\hat{\mathbf{A}}=\hat{\mathbf{B}}\hat{\mathbf{C}}^*$, then $\hat{\mathbf{B}}\hat{\mathbf{C}}^*$ is called the rank-k decomposition of $\hat{\mathbf{A}}$.
\end{definition}

If a dual quaternion matrix admits a rank-k decomposition, we refer to the rank of this matrix as $k$. 

The following lemma clarifies the correlation between the rank of a dual quaternion matrix and the rank of its dual complex adjoint matrix.
\begin{lemma}
The rank of $\hat{\mathbf{A}}\in \hat{\mathbb{Q}}^{m\times n}$ is $k$ if and only if the rank of $\mathcal{J}(\hat{\mathbf A})$ is $2k$.
\end{lemma}
\begin{proof}
By Definition \ref{de2.13} and Definition \ref{de3.6}, the rank of $\hat{\mathbf{A}}$ is $k$ if and only if there exist decomposition $\hat{\mathbf{A}}=\hat{\mathbf{B}}\hat{\mathbf{C}}^*$, where $\hat{\mathbf{B}}=\tilde{\mathbf B}_1+\tilde{\mathbf B}_2\varepsilon\in \hat{\mathbb{Q}}^{m\times k}$, $\hat{\mathbf{C}}=\tilde{\mathbf C}_1+\tilde{\mathbf C}_2\varepsilon\in \hat{\mathbb{Q}}^{n\times k}$, and the rank of $J(\tilde{\mathbf{B}}_1)$ and $J(\tilde{\mathbf{C}}_1)$ are $2k$.
It follows from (ii) and (iv) in Lemma \ref{lemma3.2} and the invertibility of mapping $\mathcal{J}$ that the rank of $\hat{\mathbf{A}}$ is $k$ if and only if there exist decomposition $\mathcal{J}(\hat{\mathbf{A}})=\mathcal{J}(\hat{\mathbf{B}})\mathcal{J}(\hat{\mathbf{C}})^*$, and the standard parts of $\mathcal{J}(\hat{\mathbf{B}})$ and $\mathcal{J}(\hat{\mathbf{C}})$ are full rank. Hence the rank of $\hat{\mathbf{A}}$ is $k$ if and only if the rank of $\mathcal{J}(\hat{\mathbf A})$ is $2k$.
\end{proof}

Utilizing the properties of the dual complex adjoint matrix, we derive the explicit solutions for the optimal rank‑k approximation of dual quaternion Hermitian matrices under both the $F$-norm and $F^*$-norm. To this end, we first relate the $F$-norm and $F^*$-norm of a dual quaternion matrix to the corresponding norms of its dual complex adjoint matrix.

\begin{lemma}\label{lemma3.5}
Let $\hat{\mathbf Q}=\tilde{\mathbf Q}_{st} +\tilde{\mathbf Q}_{\mathcal I}\varepsilon\in\hat{ \mathbb{Q}}^{n\times n}$, then it holds
\begin{equation*}
\| \mathcal{J}(\hat{\mathbf Q}) \|_{F}^2=2\| \hat{\mathbf Q}  \|_{F}^2~,~\| \mathcal{J}(\hat{\mathbf Q}) \|_{F^*}^2=2\| \hat{\mathbf Q} \|_{F^*}^2.
\end{equation*}
\end{lemma}
\begin{proof}
When $\tilde{\mathbf Q}_{st}\ne \tilde{\mathbf O}$,
suppose that $\tilde{\mathbf Q}_{st}=Q_{st,1}+Q_{st,2} j$ and $\tilde{\mathbf Q}_{\mathcal I}=Q_{\mathcal I,1}+Q_{\mathcal I,2}j$.
Denote $\hat{Q}_1=Q_{st,1}+Q_{\mathcal I,1}\varepsilon$ and $\hat{Q}_2=Q_{st,2}+Q_{\mathcal I,2}\varepsilon$. 
We have
\begin{align*} 
\| \hat{\mathbf Q} \|_{F}^2&=\| \hat{Q}_1+\hat{Q}_2 j \|_{F}^2
=\operatorname{tr}((\hat{Q}^*_1-\hat{Q}^T_2 j)(\hat{Q}_1+\hat{Q}_2 j))
\\&=\operatorname{tr}(\hat{Q}^*_1\hat{Q}_1+\hat{Q}^*_1\hat{Q}_2 j-\hat{Q}^T_2\overline{\hat{Q}_1} j+\hat{Q}^T_2\overline{\hat{Q}_2})
\\&=\operatorname{tr}(\hat{Q}^*_1\hat{Q}_1)+tr(\hat{Q}^T_2\overline{\hat{Q}_2})
=\| \hat{Q}_1 \|_{F}^2+\| \hat{Q}_2 \|_{F}^2.
\end{align*}
Therefore,
\begin{align*} 
\| \mathcal{J}(\hat{\mathbf Q})  \|_{F}^2&=\left  \| \begin{bmatrix}
\hat{Q}_1  &\hat{Q}_2\\
-\overline{\hat{Q}_2}   & \overline{\hat{Q}_1}
\end{bmatrix} \right \|_{F}^2 
\\&=\| \hat{Q}_1 \|_{F}^2+\| \hat{Q}_2 \|_{F}^2+
\| \overline{\hat{Q}_1}  \|_{F}^2+ \| \overline{\hat{Q}_2}  \|_{F}^2
\\&=2\| \hat{Q}_1 \|_{F}^2+2 \| \hat{Q}_2 \|_{F}^2
=2 \| \hat{\mathbf Q} \|_{F}^2.
\end{align*}
Furthermore,
\begin{align*} 
\| \mathcal{J}(\hat{\mathbf Q})\|_{F^*}^2&=\left  \| \begin{bmatrix}
 Q_{st,1}  & Q_{st,2}\\
-\overline{Q_{st,2}}   & \overline{Q_{st,1 }}
\end{bmatrix}
+\begin{bmatrix}
Q_{\mathcal I,1}  & Q_{\mathcal I,2}\\
-\overline{Q_{\mathcal I,2}}   & \overline{Q_{\mathcal I,1 }}
\end{bmatrix}\varepsilon \right \|_{F^*}^2 
\\&=\left  \| \begin{bmatrix}
Q_{st,1}  &  Q_{st,2}\\
-\overline{Q_{st,2}}   & \overline{Q_{st,1 }}
\end{bmatrix} \right \|_{F}^2 + \left 
\|\begin{bmatrix}
Q_{\mathcal I,1}  & Q_{\mathcal I,2}\\
-\overline{Q_{\mathcal I,2}}   & \overline{Q_{\mathcal I,1 }}
\end{bmatrix} \right \|_{F}^2 \varepsilon 
\\&=2( \| \tilde{Q}_{st} \|_{F}^2+\| \tilde{ Q}_{\mathcal I} \|_{F}^2
\varepsilon ) 
=2\| \hat{\mathbf Q}\|_{F^*}^2.
\end{align*}

When $\tilde{\mathbf Q}_{st}= \tilde{\mathbf O}$, we have
$\| \mathcal{J}(\hat{\mathbf Q}) \|_{F}^2=2\| \hat{\mathbf Q}  \|_{F}^2=\| \mathcal{J}(\hat{\mathbf Q}) \|_{F^*}^2=2\| \hat{\mathbf Q} \|_{F^*}^2=\hat{0}.$
\end{proof}

Using Lemma \ref{lemma2F} and Lemma \ref{lemma3.5}, we obtain the optimal rank-k approximation of a dual quaternion Hermitian matrix under $F$-norm.

\begin{theorem}\label{theorem3.5}
Let $\hat{\mathbf Q}=\tilde{\mathbf Q}_{st} +\tilde{\mathbf Q}_{\mathcal I}\varepsilon\in\hat{\mathbb{H}}^n$, $\{\hat{\lambda}_{i}\}_{i=1}^{n}$ are eigenvalues of $\hat{\mathbf Q}$, satisfying $|\hat{\lambda}_{1}|\ge  |\hat{\lambda}_{2}|\ge \cdots \ge  |\hat{\lambda}_{n}|$, and $\{ \hat{\mathbf v}^{i}\}^{n}_{i=1}$ are corresponding orthogonal eigenvectors with unit 2-norm. Let $\hat{\Sigma}=\operatorname{diag}\left(\hat{\lambda}_1,\hat{\lambda}_{2},\cdots,\hat{\lambda}_{k},\hat{\lambda}_1,\hat{\lambda}_{2}, \cdots , \hat{\lambda}_{k}\right)$, $\hat{\mathbf V}=[\hat{\mathbf v}^{1} ,\cdots,\hat{\mathbf v}^{k}]$, $\hat{U}=\mathcal{J}(\hat{\mathbf V})$,
$\hat{W}=\hat{U}\hat{\Sigma}\hat{U}^*$, $\hat{\mathbf Z}=\mathcal{J}^{-1}(\hat{W})$, then
$\hat{\mathbf Z}=\sum_{i=1}^{k}\hat{\lambda}_i\hat{\mathbf v}^{i} (\hat{\mathbf v}^{i}) ^*$ is the optimal rank-k approximation of $\hat{\mathbf Q}$ under $F$-norm.
\end{theorem}
\begin{proof}
Let $\hat{P}=\mathcal{J}(\hat{\mathbf Q})$. It follows from Corollary \ref{corollary3.1} that the $i$-th column and $k+i$-th column of $\hat{U}$ are eigenvectors of $\hat{P}$ with respect to eigenvalue $\lambda_i$. Suppose that $\hat{\mathbf V}=\hat{V}_1+\hat{V}_2 j$, where $\hat{V}_1$ and $\hat{V}_2$ are dual complex matrix.
Since $\{\hat{\mathbf v}^{i}\}^{k}_{i=1}$ are orthogonal, then $\hat{\mathbf V}^*\hat{\mathbf V}=\hat{\mathbf I}_k$. It follows from (i), (ii) and (iv) in Lemma \ref{lemma3.2} that $$\hat{U}^*\hat{U}=\mathcal{J}(\hat{\mathbf V})^*\mathcal{J}(\hat{\mathbf V})=\mathcal{J}(\hat{\mathbf V}^*\hat{\mathbf V})=
\mathcal{J}(\hat{\mathbf I}_k)=\hat{I}_{2k}.$$
Then by Lemma \ref{lemma2F}, $\hat{W}$ is the optimal rank-2k approximation of $\hat{P}$ under $F$-norm, i.e., $\| \hat{P}-\hat{W} \|^2_{F}\le \| \hat{P}-\hat{S} \|^2_{F}$ holds for any dual complex matrix $\hat{S}$ with rank 2k. Hence, by Lemma \ref{lemma3.5}, for any rank-k approximation $\hat{\mathbf R}$ of $\hat{\mathbf Q}$, it holds 
$$\| \hat{\mathbf Q}-\hat{\mathbf R} \|^2_{F}=\frac{1}{2} 
\| \hat{P}-\mathcal{J}(\hat{\mathbf R})\|^2_{F}\ge
\frac{1}{2} \| \hat{P}-\hat{W} \|^2_{F}
= \| \hat{\mathbf Q}-\hat{\mathbf Z} \|^2_{F}.$$
Thus $\hat{\mathbf Z}$ is the optimal rank-k approximation of $\hat{\mathbf Q}$ under $F$-norm. 

Let $\Sigma=\operatorname{diag}(\hat{\lambda}_1,\hat{\lambda}_{2}, \cdots , \hat{\lambda}_{k})$.
Since 
\begin{align*} 
\hat{U}\hat{\Sigma}\hat{U}^*&=
\begin{bmatrix}
\hat{V}_1  & \hat{V}_2\\
-\overline{\hat{V}_2}   & \overline{\hat{V}_1}
\end{bmatrix}
\begin{bmatrix}
\Sigma  & \\
  &\Sigma
\end{bmatrix}
\begin{bmatrix}
\overline{\hat{V}_1}^T  & -\hat{V}_2^T   \\
 \overline{\hat{V}_2}^T & \hat{V}_1^T
\end{bmatrix}
\\&=\begin{bmatrix}
\hat{ V}_1 \Sigma \overline{\hat{ V}_1}^T+ \hat{ V}_2 \Sigma \overline{\hat{ V}_2}^T & -\hat{ V}_1 \Sigma\hat{ V}_2^T+ \hat{ V}_2 \Sigma\hat{ V}_1^T \\
-\overline{\hat{ V}_2} \Sigma\overline{\hat{ V}_1}^T+
\overline{\hat{ V}_1} \Sigma\overline{\hat{ V}_2}^T  & \overline{\hat{ V}_2} \Sigma\hat{ V}_2^T+\overline{\hat{ V}_1} \Sigma\hat{ V}_1^T
\end{bmatrix},
\end{align*}
then
\begin{align*}
\hat{\mathbf Z}=\mathcal{J}^{-1}(\hat{\mathbf W})&=\hat{V}_1\Sigma \overline{\hat{V}_1}^T+ \hat{V}_2 \Sigma \overline{\hat{V}_2}^T +(-\overline{\hat{V}_2}\Sigma\overline{\hat{V}_1}^T+
\overline{\hat{V}_1}\Sigma\overline{\hat{V}_2}^T) j
\\&=\hat{\mathbf V}\hat{\Sigma}\hat{\mathbf V}^*=\sum_{i=1}^{k}\hat{\lambda}_i\hat{\mathbf v}^{i} (\hat{\mathbf v}^{i})^* .
\end{align*}
\end{proof}

Pursuant to Theroem \ref{theorem3.5}, we introduce Algorithm \ref{rank-1-F-approximation} for computing the optimal rank-k approximation of dual quaternion Hermitian matrices under $F$-norm.

\begin{algorithm}
    \caption{The optimal rank-k approximation of dual quaternion Hermitian matrices under $F$-norm}\label{rank-1-F-approximation}
    \begin{algorithmic}
    \REQUIRE $\hat{\mathbf Q}\in \hat{\mathbb H}^n$.
    \ENSURE $\hat{\mathbf V}$.
    \STATE  Let $\hat{\mathbf V}=\hat{\mathbf O}$.
    \FOR{$t=1:k$}
    \STATE  Compute the dominate eigenvalue $\hat{\lambda}_t$ and corresponding eigenvector $\hat{\mathbf v}_t\in \hat{\mathbb{Q}}_2^{n\times1}$ of $\hat{\mathbf Q}$ by Algorithm \ref{Power_method_algorithm}.
    \STATE Compute $\hat{\mathbf Q}=\hat{\mathbf Q}-\hat{\lambda}_t\hat{\mathbf v}_t\hat{\mathbf v}_t^*$. $\hat{\mathbf V}=\hat{\mathbf V}+\hat{\lambda}_t\hat{\mathbf v}_t\hat{\mathbf v}_t^*$.
    \ENDFOR
    \end{algorithmic}
\end{algorithm}

Next, utilizing Lemma \ref{lemma2F*} and Lemma \ref{lemma3.5}, we obtain the optimal rank-k approximation of a dual quaternion Hermitian matrix under $F^*$-norm.
\begin{theorem}\label{F*}
Let $\hat{\mathbf Q}=\tilde{\mathbf Q}_{st} +\tilde{\mathbf Q}_{\mathcal I}\varepsilon\in \hat{\mathbb H}^n$, $\{ \lambda_{i}\}_{i=1}^{n}$ are eigenvalues of $\tilde{\mathbf Q}_{st}$, satisfying $|\lambda_{1}|\ge |\lambda_{2}|\ge \cdots \ge |\lambda_{n}|$, $\{\tilde{\mathbf v}^{i}\}^{n}_{i=1}$ are corresponding orthogonal eigenvectors with unit 2-norm. Suppose that $\tilde{\mathbf v}^{i} =\mathbf v^{i}_{1}+\mathbf v^{i} _{2} j$, 
$\hat{P}=\mathcal{J}(\hat{\mathbf Q})=P_1+P_2\varepsilon$,
$\Sigma=\operatorname{diag}(\lambda_1,\lambda_{2}, \cdots , \lambda_{k},\lambda_1,\lambda_{2}, \cdots , \lambda_{k})$,
$V=\begin{bmatrix}
V_1  & V_2\\
-\overline{V_2}   & \overline{V_1}
\end{bmatrix}$, where $V_1=[\mathbf v^{1} _{1},\cdots,\mathbf v^{k}_{1}]$, $V_2=[ \mathbf v^{1} _{2},\cdots,\mathbf v^{k} _{2}]$.
Denote \begin{equation}\label{kf*}
\hat{U}=V\Sigma V^*+\left(P_2-(I_{2n}- V V^*) P_2( I_{2n}- V V^*)\right)\varepsilon,
\end{equation}
then $\hat{\mathbf W}=\mathcal{J}^{-1}(\hat{U})$ is the optimal rank-k approximation of $\hat{\mathbf Q}$ under $F^*$-norm.
\end{theorem}
\begin{proof}
Since $\tilde{\mathbf v}^{i}$ is the eigenvector of $\tilde{\mathbf Q}_{st}$ with respect to eigenvalue
$\lambda_{i}$, it follows from Lemma \ref{corollary3.1} that 
$\begin{bmatrix}
\mathbf v^{i}_{1} \\
-\overline{\mathbf v^{i}_{2}}
\end{bmatrix}$ and $\begin{bmatrix}
\mathbf v^{i}_{2} \\
\overline{\mathbf v^{i}_{1}}
\end{bmatrix}$ are eigenvectors of $P_1$ with respect to
eigenvalue $\lambda_{i}$. Since $\left \{ \tilde{\mathbf v}^{i} \right \} ^{k}_{i=1}$ are orthogonal, then
\begin{align*} 
\tilde{\mathbf I}_k&=(\overline{V_1}-V_2 j)^T( V_1+V_2 j)=\overline{ V_1}^TV_1+V_2^T\overline{V_2}+
(\overline{V_1}^T V_2- V_2^T\overline{V_1}) j.
\end{align*}
Hence,
\begin{align*} 
V^*V&=\begin{bmatrix}
\overline{V_1}^T  & -V_2^T   \\
\overline{V_2}^T & V_1^T
\end{bmatrix}
\begin{bmatrix}
V_1  & V_2\\
-\overline{V_2}   & \overline{V_1}
\end{bmatrix}
=I_{2k}.
\end{align*}
Therefore, $V\Sigma V^*$ is the optimal rank-2k approximation of $P_1$ under $F$-norm.
It follows from Lemma \ref{lemma2F*} that $\hat{U}$ is the optimal rank-2k approximation of $\hat{P}$ under $F^*$-norm, i.e., it holds $\| \hat{P}-\hat{U} \|^2_{F^*}\le \| \hat{P}-\hat{S}\|^2_{F^*}$ for any dual complex Hermitian matrix $\hat{S}$ with rank 2k. By direct verification, $\hat{U}\in DM(\mathbb{C}^{n\times n})$.
Hence, by Lemma \ref{lemma3.5}, for any rank-k approximation $\hat{\mathbf R}$ of $\hat{\mathbf Q}$, it holds 
$$\| \hat{\mathbf Q}-\hat{\mathbf R}\|^2_{F^*}=\frac{1}{2} 
\| \hat{P}-\mathcal{J}(\hat{\mathbf R}) \|^2_{F^*}\ge
\frac{1}{2} \| \hat{P}-\hat{U} \|^2_{F^*}
=\| \hat{\mathbf Q}-\hat{\mathbf W} \|^2_{F^*}.$$
Therefore, $\hat{\mathbf W}$ is the optimal rank-k approximation of $\hat{\mathbf Q}$ under $F^*$-norm.
\end{proof}

By Lemma \ref{F*}, we introduce algorithm \ref{rank-k-F*-approximation} for computing the optimal rank-k approximation of dual quaternion Hermitian matrices under $F^*$-norm.

\begin{algorithm}
    \caption{The optimal rank-k approximation of dual quaternion Hermitian matrices under the $F^*$-norm}\label{rank-k-F*-approximation}
    \begin{algorithmic}
    \REQUIRE $\hat{\mathbf Q}\in \hat{\mathbb H}^n$.
    \ENSURE $\hat{\mathbf W}$.
    \STATE Compute $\hat{P}=P_1+P_2\varepsilon=\mathcal{J}(\hat{\mathbf Q})$.
    \STATE Compute the largest k eigenvalues $\{\lambda_{i}\}_{i=1}^{k}$
    and corresponding orthogonal eigenvectors $\{\tilde{\mathbf v}^{i}=\mathbf v^{i}_{1}+\mathbf v^{i} _{2} j\}^{k}_{i=1}$ with unit 2-norm of $\tilde{\mathbf Q}_{st}$ by power method.
    \STATE Let $\Sigma=\operatorname{diag}(\lambda_1,\lambda_{2}, \cdots , \lambda_{k},\lambda_1,\lambda_{2}, \cdots , \lambda_{k})$,
    $V_1=[\mathbf v^{1} _{1},\cdots,\mathbf v^{k}_{1}]$, $V_2=[ \mathbf v^{1} _{2},\cdots,\mathbf v^{k} _{2}]$, and 
    $V=\begin{bmatrix}
    V_1  & V_2\\
    -\overline{V_2} & \overline{V_1}
    \end{bmatrix}$.
    \STATE Compute $\hat{U}=V\Sigma V^*+\left(P_2-(I_{2n}- V V^*) P_2( I_{2n}- V V^*)\right)\varepsilon$.
    \STATE Compute $\hat{\mathbf W}=\mathcal{J}^{-1}(\hat{U})$.
    \end{algorithmic}
\end{algorithm}

For the special case $k=1$, to further reduce the computational cost of computing the optimal rank‑1 approximation of a dual quaternion Hermitian matrix under the $F^*$-norm, we propose the following efficient algorithm (\ref{rank-1-F*-approximation}).

\begin{theorem}\label{3.2}
Let $\hat{\mathbf Q}=\tilde{\mathbf Q}_{st} +\tilde{\mathbf Q}_{\mathcal I}\varepsilon\in \hat{\mathbb H}^n$, $\hat{P}=P_1+P_2\varepsilon=\mathcal{J}(\hat{\mathbf Q})$. Suppose that $\lambda$ is the dominate eigenvalue of $P_1$, and 
$\mathbf v=[\mathbf v_1^T ~ \mathbf v_2^T]^T$ is the corresponding unit eigenvector, where $\mathbf v_1,\mathbf v_2\in \mathbb C^{n\times 1}$. 
Denote $\Sigma_1=\operatorname{diag}(\lambda,\lambda)$, 
$V=\begin{bmatrix}
\mathbf v_1 &-\overline{\mathbf v_2} \\
\mathbf v_2 &\overline{\mathbf v_1}
\end{bmatrix}$, and $S= V^* P_2 V$.
Suppose that $\mu$ is the eigenvalue of $S$, and 
$\mathbf u=[\mathbf u_1 ~ \mathbf u_2]^T\in \mathbb C^{2\times 1}$ is the corresponding unit eigenvector. Denote $\Sigma_2=\operatorname{diag}(\mu,\mu)$, and
$U=\begin{bmatrix}
\mathbf u_1 &-\overline{\mathbf u_2} \\
\mathbf u_2 &\overline{\mathbf u_1}
\end{bmatrix}$.
Denote $V_1=U_1=V U$ and $V_2=U_2=\frac{1}{\lambda}( P_2 V_1- V_1\Sigma_2)$. Let $\hat{V}=\hat{U}= V_1+ V_2\varepsilon$, $\hat{\Sigma}=\Sigma_1+\Sigma_2\varepsilon$, then $\hat{U}\hat{\Sigma}\hat{V}^*$ is the optimal rank-2 approximation of $\hat{P}$ under $F^*$-norm. 
Let $\hat{\mathbf u}=\mathcal{J}^{-1}(\hat{U})$, then 
$(\lambda+\mu\varepsilon)\hat{\mathbf u}(\hat{\mathbf u})^*$ is the optimal rank-1 approximation of $\hat{\mathbf Q}$ under $F^*$-norm.
\end{theorem}
\begin{proof}
By Corollary \ref{corollary3.1}, the vector $[-\overline{\mathbf v_2}^T~\overline{\mathbf v_1}^T]^T$ is also a unit eigenvector of $P_1$ corresponding to eigenvalue $\lambda$, and is orthogonal to $\mathbf v$. Hence $V \Sigma_1 V^*$ is the optimal rank-2 approximation of $P_1$. Similarly, Corollary \ref{corollary3.1} implies that $[-\overline{\mathbf u_2}~\overline{\mathbf u_1}]^T$ is a unit eigenvector of $S$ associated with eigenvalue $\mu$, orthogonal to $\mathbf u$. It follows that $U$ is unitary. Therefore, $V_1 \Sigma_1 V_1^*= V U \Sigma_1 U^* V^*= V\Sigma_1 V^*$. Hence, $ V_1 \Sigma_1 V_1^*$ is also the optimal rank-2 approximation of $P_1$.

Let $\hat{A}=\hat{P}$, and using the above $\Sigma_1, V_1, U_1$, in Lemma \ref{lemma2F*}, we have $\Delta= O^{2\times 2}$ and $R=U_1A_2V_1^*=\Sigma_2$ is a diagonal matrix, then $\Psi= O$ and $\frac12\operatorname{diag}(\mathrm{sym}(R))=\Sigma_2$. Take $\Omega= O$, then
\begin{align*} 
V_2=U_2&=U_1\left(\mathrm{sym}(R \Sigma_1)
\odot \Delta+\Omega+\Psi\right)+
(I_m- U_1U_1^*)A_2V_1
\Sigma_1^{-1}
\\&=(I_m-U_1U_1^*)A_2V_1
\Sigma_1^{-1}
=\frac{1}{\lambda} (A_2V_1-U_1U_1^*A_2V_1)
\\&=\frac{1}{\lambda} (A_2V_1-U_1\Sigma_2)=\frac{1}{\lambda} (P_2V_1-U_1\Sigma_2).
\end{align*}
Then by Lemma \ref{lemma2F*}, $\hat{U}\hat{\Sigma}\hat{V}^*$ is the optimal rank-2 approximation of $\hat{P}$ under $F^*$-norm. By direct calculation, it holds $\mathcal J^{-1}(\hat{U}\hat{\Sigma}\hat{V}^*)=(\lambda+\mu\varepsilon)\hat{\mathbf u}(\hat{\mathbf u})^*$. Similarly to the proof of Theorem \ref{F*}, $(\lambda+\mu\varepsilon)\hat{\mathbf u}(\hat{\mathbf u})^*$ is likewise the optimal rank-1 approximation of $\hat{\mathbf Q}$ under the $F^*$-norm.
\end{proof}

\begin{algorithm}
    \caption{An efficient algorithm to compute the optimal rank-1 approximation of dual quaternion Hermitian matrices under the $F^*$-norm}\label{rank-1-F*-approximation}
    \begin{algorithmic}
    \REQUIRE $\hat{\mathbf Q}\in \hat{\mathbb H}^n$.
    \ENSURE $\hat{\mathbf V}$.
    \STATE Compute $\hat{P}=P_1+P_2\varepsilon=\mathcal{J}(\hat{\mathbf Q})$.
    \STATE Compute the dominate eigenvalue $\lambda$ of $P_1$ and corresponding unit eigenvector $\mathbf v=[\mathbf v_1^T~\mathbf v_2^T]^T$ by power method. 
    \STATE Let $V=\begin{bmatrix}
    \mathbf v_1 &-\overline{\mathbf v_2} \\
    \mathbf v_2 &\overline{\mathbf v_1}
    \end{bmatrix}$, compute $R=V^*P_2V$.
    \STATE Compute the eigenvalue $\mu$ and the corresponding unit eigenvector $\mathbf u=[\mathbf u_1~\mathbf u_2]^T$ of the $2\times 2$ matrix $R$ .
    \STATE Let $U=\begin{bmatrix}
    \mathbf u_1 &-\overline{\mathbf u_2} \\
    \mathbf u_2 &\overline{\mathbf u_1}
    \end{bmatrix}$, and $\Sigma=\operatorname{diag}(\mu,\mu)$.
    \STATE Compute $V_1=VU$, $V_2=\frac{1}{\lambda} (P_2W_1- W_1\Sigma)$ and $\hat{\mathbf u}=\mathcal{J}^{-1}(V_1+V_2\varepsilon)$.
    \STATE Compute $\hat{\mathbf V}=(\lambda+\mu\varepsilon)\hat{\mathbf u}(\hat{\mathbf u})^*$.
    \end{algorithmic}
\end{algorithm}

Since computing the dual term $$W_2=\frac{1}{\lambda} (P_2W_1- W_1\Sigma)$$ in Algorithm \ref{rank-1-F*-approximation} is far more succinct than computing the dual term $$P_2-(I_{2n}- V V^*) P_2( I_{2n}- V V^*)=V V^*P_2V V^*-V V^*P_2-P_2V V^*$$ in Algorithm \ref{rank-k-F*-approximation} for $k=1$, this succinctness highlights the clear computational advantage of Algorithm \ref{rank-1-F*-approximation}.

\section{The improved two-block coordinate descent method}\label{section5}

Recently, Cui and Qi \cite{cui2024power} introduced a two-block coordinate descent method to solve the PGO problem under the $F$-norm. Building on the results from the previous sections, we now reformulate the PGO problem using the more appropriate and robust $F^*$-norm and  tackle it within the same two-block coordinate descent framework. In addition, we augment the two-block coordinate descent method with several practical enhancements, including proper parameter selection, stagnation-based termination criteria, and an effective spectral initialization strategy.

We have provided an example in Section \ref{Preliminaries} to demonstrate that the $F^*$-norm more faithfully captures the magnitude of the dual part of a dual quaternion matrix than the $F$-norm. The following lemma further offers the dual part of the $F$-norm an upper bound.

\begin{lemma}\label{ll}
Let $\hat{\mathbf{Q}}=\tilde{\mathbf{Q}}_{st}+\tilde{\mathbf{Q}}_{\mathcal I}\varepsilon\in\hat{\mathbb Q}^{n\times m}$. Then $$|\|\hat{\mathbf Q}\|_{F,\mathcal I}|\le\|\tilde{\mathbf Q}_{\mathcal I}\|_F.$$
\end{lemma}
\begin{proof}
By definition, when $\tilde{\mathbf{Q}}_{st}\ne \tilde{\mathbf O}$, 
\begin{align*} 
|\|\hat{\mathbf Q}\|_{F,\mathcal I}|&=\left | \frac{sc (tr(\tilde{\mathbf Q}_{st}^{\ast }\tilde{\mathbf Q}_{\mathcal I}))}{\|\tilde{\mathbf Q}_{st} \|_{F}} \right | \le \frac{\sum_{ij}^{}|\tilde{\mathbf{Q}}_{st,ij}||\tilde{\mathbf{Q}}_{\mathcal I,ij}| }{\|\tilde{\mathbf Q}_{st} \|_{F}}
\le \frac{\|\tilde{\mathbf{Q}}_{st}\|_F \|\tilde{\mathbf{Q}}_{\mathcal I}\|_F}{\|\tilde{\mathbf{Q}}_{st}\|_F}
=\|\tilde{\mathbf{Q}}_{\mathcal I}\|_F.
\end{align*}
When $\tilde{\mathbf{Q}}_{st}= \tilde{\mathbf O}$, $|\|\hat{\mathbf Q}\|_{F,\mathcal I}|=\|\tilde{\mathbf Q}_{\mathcal I}\|_F$.
\end{proof}

Lemma \ref{ll} demonstrates that $\|\hat{\mathbf Q}\|^2_{F,\mathcal I}\le \|\tilde{\mathbf Q}_{\mathcal I}\|^2_F$. Note that $$\|\hat{\mathbf Q}\|^2_{F^*}=\|\tilde{\mathbf Q}_{st}\|^2_{F}+\|\tilde{\mathbf Q}_{\mathcal I}\|^2_{F}\varepsilon.$$
This also shows that using $F^*$-norm is both natural and advantageous.

Now, we consider reformulating the PGO problem under the $F^*$-norm, i.e.,
\begin{equation}
\min_{\hat{\mathbf{x}}\in\hat{\mathbb{U}}^{n\times1}}\|P_E(\hat{\mathbf{Q}}-\hat{\mathbf{x}}\hat{\mathbf{x}}^*)\|_{F^*}^2.
\end{equation}
Similarly to the analysis in \cite{cui2024power}, we transform the problem into solving the following problem:
\begin{equation}
\begin{aligned}
\min_{\hat{\mathbf{X}}_1,\hat{\mathbf{X}}_2}&\frac12\|P_E(\hat{\mathbf{X}}_1-\hat{\mathbf{Q}})\|_{F^*}^2+\frac\rho2\|\hat{\mathbf{X}}_1-\hat{\mathbf{X}}_2\|_{F^*}^2,
\\\mathrm{s.t.}&~\hat{\mathbf{X}}_1\in C_1,\hat{\mathbf{X}}_2\in C_2.
\end{aligned}
\end{equation}

Then, employing the two-block coordinate descent method, at each iteration $k$, we alternately compute $\hat{\mathbf{X}}_1^{(k)}$ and $\hat{\mathbf{X}}_2^{(k)}$ by solving the following subproblems:
\begin{align}
\hat{\mathbf{X}}_1^{(k)}&=\underset{\hat{\mathbf{X}}_1\in C_1}{\operatorname*{argmin}}
\frac12\|P_E(\hat{\mathbf{X}}_1-\hat{\mathbf{Q}})\|_{F^*}^2+\frac{\rho^{(k-1)}}2\|\hat{\mathbf{X}}_1-\hat{\mathbf{X}}_2^{(k-1)}\|_{F^*}^2,\label{qq1}\\
\hat{\mathbf{X}}_2^{(k)}&=\underset{\hat{\mathbf{X}}_2\in C_2}{\operatorname{argmin}} 
\frac{\rho^{(k-1)}}2\|\hat{\mathbf{X}}_2
-\hat{\mathbf{X}}_1^{(k)}\|_{F^*}^2.\label{qq2}
\end{align}

Now, we show that subproblem \eqref{qq1} has a closed-form solution.

According to Proposition 2 in \cite{cui2024power}, we have the following conclusion.
\begin{lemma}\label{Lemma5.2}
Let $\hat{q}\in \hat{\mathbb{Q}}$ and $\tilde{q}_{st}\ne \tilde{0}$, then $\hat{u} = \frac{\hat{q}}{|\hat{q}|}$ is the optimal solution to the following problem:
$$\hat{u} = \underset{\hat{v} \in \hat{\mathbb{U}}}{\arg\min}~|\tilde{v}_{st} - \tilde{q}_{st}|^2 + \left| \tilde{v}_\mathcal{I} - \frac{\tilde{q}_\mathcal{I}}{|\tilde{q}_{st}|} \right|^2 \epsilon.$$
\end{lemma}

By Lemma \ref{Lemma5.2}, we can derive the following lemma.
\begin{lemma}\label{lemma5.3}
Let $\hat{q}\in \hat{\mathbb{Q}}$, denote
\begin{equation*}
\hat{u}=\begin{cases}
\frac{\tilde{q}_{st}+|\tilde{q}_{st}|\tilde{q}_{\mathcal{I}}\varepsilon}{|\tilde{q}_{st}+
|\tilde{q}_{st}|\tilde{q}_{\mathcal{I}}\varepsilon|}, &\text{if}~~\tilde{q}_{st}\ne \tilde{0},\\
\tilde{c}+\tilde{q}_{\mathcal{I}}\varepsilon\in \hat{\mathbb U}, & \text{if}~~\tilde{q}_{st}= \tilde{0},
\end{cases}
\end{equation*}
then $\hat{u}$ is the optimal solution to 
$$\hat{u} = \underset{\hat{v} \in \hat{\mathbb{U}}}{\arg\min}~|\tilde{v}_{st} - \tilde{q}_{st}|^2 + | \tilde{v}_\mathcal{I} - \tilde{q}_\mathcal{I}|^2 \epsilon.$$
\end{lemma}
\begin{proof}
When $\tilde{q}_{st}\ne \tilde{0}$, then by Lemma \ref{Lemma5.2}, $\hat{u}$ is the optimal solution to
\begin{align*} 
\hat{u} &= \underset{\hat{v} \in \hat{\mathbb{U}}}{\arg\min}~|\tilde{v}_{st} - 
\tilde{q}_{st}|^2 + \left| \tilde{v}_\mathcal{I} - \frac{|\tilde{q}_{st}|\tilde{q}_{\mathcal{I}}}{|\tilde{q}_{st}|} \right |^2 \epsilon
\\&=\underset{\hat{v} \in \hat{\mathbb{U}}}{\arg\min}~|\tilde{v}_{st} - 
\tilde{q}_{st}|^2 + | \tilde{v}_\mathcal{I} - \tilde{q}_{\mathcal{I}}|^2 \epsilon.
\end{align*}
When $\tilde{q}_{st}=\tilde{0}$, 
$$|\tilde{v}_{st}-\tilde{q}_{st}|^2 + | \tilde{v}_\mathcal{I} - \tilde{q}_{\mathcal{I}}|^2 \epsilon
=|\tilde{v}_{st}|^2 + | \tilde{v}_\mathcal{I} - \tilde{q}_{\mathcal{I}}|^2 \epsilon
\ge 1+| \tilde{v}_\mathcal{I} - \tilde{q}_{\mathcal{I}}|^2 \epsilon \ge 1,$$
then for any quaternion $\tilde{c}$ satisfying $\tilde{c}+\tilde{q}_{\mathcal{I}}\varepsilon\in \hat{\mathbb U}$ is the the optimal solution to
$$\hat{u} = \underset{\hat{v} \in \hat{\mathbb{U}}}{\arg\min}~|\tilde{v}_{st} |^2 + | \tilde{v}_\mathcal{I} - \tilde{q}_\mathcal{I}|^2 \epsilon.$$
\end{proof}

By \eqref{qq1}, for $i\ne j$, 
\begin{align*} 
(\hat{x}_{1,ij}^{(k)},\hat{x}_{1,ji}^{(k)})=\underset{\hat{u},\hat{v} 
\in \hat{\mathbb{U}},\hat{u}=\hat{v}^*}{\arg\min}~
\frac{1}{2}&(\delta_{E,ij}\|\hat{u}-\hat{q}_{ij}\|^2_{F^*}+
\delta_{E,ji}\|\hat{v}-\hat{q}_{ij}\|^2_{F^*} \\&+
\rho^{(k-1)}\|\hat{u}-\hat{x}_{2,ij}^{(k-1)}\|^2_{F^*} +
\rho^{(k-1)}\|\hat{v}-\hat{x}_{2,ji}^{(k-1)}\|^2_{F^*})
\\=\underset{\hat{u} \in \hat{\mathbb{U}},\hat{v}=\hat{u}^*}{\arg\min}~
\frac{1}{2}&(\delta_{E,ij}\|\hat{u}-\hat{q}_{ij}\|^2_{F^*}+
\delta_{E,ji}\|\hat{u}-\hat{q}^*_{ij}\|^2_{F^*} \\&+
\rho^{(k-1)}\|\hat{u}-\hat{x}_{2,ij}^{(k-1)}\|^2_{F^*} +
\rho^{(k-1)}\|\hat{u}-(\hat{x}_{2,ji}^{(k-1)})^*\|^2_{F^*})
\\=\underset{\hat{u} \in \hat{\mathbb{U}},\hat{v}=\hat{u}^*}{\arg\min}~
\frac{1}{2}&(\delta_{E,ij}+\delta_{E,ji}+2\rho^{(k-1)})
\|\hat{u}-\hat{q}\|^2_{F^*},
\end{align*}
where
$$\hat{q}= \frac{\delta_{E,ij}
\hat{q}_{ij}+\delta_{E,ji}\hat{q}^*_{ji}+\rho^{(k-1)}(
\hat{x}_{2,ij}^{(k-1)}+(\hat{x}_{2,ji}^{(k-1)})^*)
}{\delta_{E,ij}+\delta_{E,ji}+2\rho^{(k-1)}}.$$ Here, we simply denote $\|\hat{p}\|^2_{F^*}=
|\tilde{p}_{st}|^2+|\tilde{p}_{\mathcal{I}}|^2\varepsilon$, for any $\hat{p}\in \hat{\mathbb{Q}}$.

Then by Lemma \ref{lemma5.3}, subproblem \eqref{qq1} has an explicit solution
\begin{equation}\label{e5.6}
\hat{x}_{1,ij}^{(k)}=\begin{cases}
\hat{1} ,&\text{if}~~i=j,\\
\frac{\tilde{q}_{st}+|\tilde{q}_{st}|\tilde{q}_{\mathcal{I}}\varepsilon}{|\tilde{q}_{st}+|\tilde{q}_{st}|\tilde{q}_{\mathcal{I}}\varepsilon|},&\text{if}~~i\ne j,~\tilde{q}_{st}\ne\tilde{0},\\
c+\tilde{q}_{\mathcal{I}}\varepsilon \in \hat{\mathbb{U}} ,&\text{if}~~i\ne j,~\tilde{q}_{st}=\tilde{0}.
\end{cases}
\end{equation} 

Note that subproblem \eqref{qq2} is the optimal rank-1 approximation of a dual quaternion Hermitian matrix under the $F^*$-norm. It follows from Theorem \ref{3.2} that subproblem \eqref{qq2} has an explicit solution that can be efficiently calculated by Algorithm \ref{rank-1-F*-approximation}.

We will further improve the two-block coordinate descent method by the following aspects.

First, since measurement noise may affect the satisfaction of original stopping criteria in \cite{cui2024power} and parameter choices can strongly affect both accuracy and convergence speed, we adopt two refinements. We apply the proper parameter selection $\rho_1=1$ and introduce a stagnation-based stopping rule
\begin{equation}\label{terminal}
\underset{k-d\le i<j\le k}{\max} |R^{(i)}-R^{(j)}|\le \beta, ~~R^{(i)}=\|\hat{\mathbf{X}}_1^{(i)}-\hat{\mathbf{X}}_2^{(i)}\|_{F^R}.
\end{equation}
With this criterion, the algorithm halts early whenever the objective value plateaus, yielding substantial computational savings without sacrificing solution quality.

\begin{algorithm}
    \caption{A spectral approach to SE(3) synchronization}\label{sass}
    \begin{algorithmic}
    \REQUIRE $\hat{\mathbf Q}$.
    \ENSURE $\hat{\mathbf{X}}$
    \STATE Compute the dominate eigenvector $\hat{\mathbf{y}}$ of $\hat{\mathbf Q}$ by Algorithm \ref{Power_method_algorithm} with tolerance $\delta$.
    \STATE Compute $\hat{\mathbf{z}}=(\hat{z}_i)$ as the projection of $\hat{\mathbf{y}}$ onto the set $\hat{\mathbb Q}^{n\times 1}_2$ by (\ref{def11}).
    \STATE Let $\hat{\mathbf{x}}=[\hat{z}_1,\hat{z}_2,\cdots,\hat{z}_n]^*$.
    \STATE Compute $\hat{\mathbf{X}}=\hat{\mathbf{x}}\hat{\mathbf{x}}^*$.
    \end{algorithmic}
\end{algorithm} 

Second, since random initialization can lead to widely varying iteration counts, we seek a more reliable starting point for the PGO problem. Drawing on the spectral SE(3) synchronization method of \cite{hadi2023mathrm} (see Algorithm \ref{sass}), we can efficiently generate a high-quality starting point $\hat{\mathbf{X}}_2^{(0)}$, as the PGO problem can be viewed as an SE(3) synchronization problem. Although this spectral initialization may not match the final accuracy of Algorithm \ref{PGOP}, its low computational complexity makes it an excellent seed. Moreover, to accelerate the most expensive step—computing the dominant eigenvector of a dual quaternion Hermitian matrix—we replace it with the efficient Algorithm \ref{Power_method_algorithm}.

Combining the above improvement strategies, we obtain an improved two-block coordinate descent method, as described in Algorithm \ref{PGOP}, to solve the PGO problem under the $F^*$-norm.

\begin{algorithm}
    \caption{Improved two-block coordinate descent method}\label{PGOP}
    \begin{algorithmic}
    \REQUIRE $\hat{\mathbf Q}$, $(V,E)$, the
    maximum iteration number $k_{max}$, parameter $\rho^{(k)}=\rho$, $d$, and tolerance $\delta$, $\beta$.
    \ENSURE $\hat{\mathbf{X}}_1^{(k)}$, $\hat{\mathbf{X}}_2^{(k)}$.
    \STATE Generate initial iteration matrix $\hat{\mathbf{X}}_2^{(0)}$ by Algorithm \ref{sass} with tolerance $\delta$.
    \FOR{$k=1,2,\cdots,k_{max}$}
    \STATE Update $\hat{\mathbf{X}}_1^{(k)}$ by (\ref{e42}) or (\ref{e5.6}).
    \STATE Update $\hat{\mathbf{X}}_2^{(k)}$ by  Algorithm \ref{Power_method_algorithm} or Algorithm \ref{rank-1-F*-approximation}.
    \STATE Compute $R^{(k)}=\|\hat{\mathbf{X}}_1^{(k)}-\hat{\mathbf{X}}_2^{(k)}\|_{F^R}$.
    \STATE If $R^{(k)}\le\beta$ or $\underset{k-d\le i<j\le k}{\max} |R^{(i)}-R^{(j)}|\le \beta$, then Stop.
    \ENDFOR
    \end{algorithmic}
\end{algorithm}

We denote the original two‑block coordinate descent method in \cite{cui2024power} by tBCD. By applying proper parameter selection and the stagnation-based termination criteria \eqref{terminal} in tBCD, we obtain RtBCD. Within Algorithm \ref{PGOP}, updating $\hat{\mathbf{X}}_1^{(k)}$ via (\ref{e42}) and updating $\hat{\mathbf{X}}_2^{(k)}$ via Algorithm \ref{Power_method_algorithm} yields the DBtBCD variant, whereas 
updating $\hat{\mathbf{X}}_1^{(k)}$ via (\ref{e5.6}) and updating $\hat{\mathbf{X}}_2^{(k)}$ via Algorithm \ref{rank-1-F*-approximation} produces the FtBCD variant.

\section{Experimental Results}\label{section6}
In this section, we report numerical experiments on the PGO problem by Algorithm \ref{PGOP} developed above. All numerical experiments are conducted in MATLAB (2022a) on a laptop of 8G of memory and Inter Core i5 2.3Ghz CPU. 

For the PGO problem with $m$ directed edges and $n$ vertices in $E$, the observation rate is defined as $s=\frac{m}{n(n-1)}$. Let $\hat{\mathbf Q}_0$ denote the true value. To generate the observed data, we first  sample Gaussian noise $\hat{\mathbf N}$ and define the relative noise level $l_{\text{noise}}=\frac{\|P_E(\hat{\mathbf{N}})\|_{F^R}}
{\|P_E(\hat{\mathbf{Q}}_0)\|_{F^R}}$. We then form the observed data $\hat{\mathbf Q}=P_E(P_{\hat{\mathrm{U}}^{n\times n}}(\hat{\mathbf{Q}}_0+\hat{\mathbf{N}}))$. Next, set $l^*=\frac{\|\hat{\mathbf{Q}}-P_E(\hat{\mathbf{Q}}_0)\|_{F^R}}
{l_{\text{noise}}\|P_E(\hat{\mathbf{Q}}_0)\|_{F^R}}$. In practice, one observes $l^*\approx 1.2$. To quantify reconstruction quality, the relative error $e_Q$ is defined as $e_{Q}=\frac{\|\hat{\mathbf{Q}}_0-\hat{\mathbf{X}}_2^{(k)}\|_{F^R}}
{\|\hat{\mathbf{Q}}_0\|_{F^R}}$. $n_{iter}$ refers to the average number of iterations required by the algorithm, while $'\text{time (s)}'$ denotes the average elapsed CPU time consumed by the algorithm.

\subsection{Parameter selection and termination condition}
First, we assess how our parameter selection and termination criteria accelerate convergence. Figure \ref{fig1} compares tBCD and RtBCD on the PGO problem with $n=100$ in two settings: (i) noiseless data at varying observation rates, and (ii) fixed observation rate $s=0.10$ with different noise levels. Table \ref{table4} summarizes the final error $e_Q$, number of iterations, and elapsed CPU time for each test. For tBCD, the parameters are set as $k_{max}=200$, $\rho^{(0)}=0.01$, $\rho_1=1.1$, $\beta=10^{-6}$. For RtBCD, the parameters are set as $k_{max}=200$, $\rho^{(0)}=0.01$, $\rho_1=1.0$, $d=1$, $\beta=10^{-6}$. 


\begin{figure}[htbp]
     \centering
    \subfigure[PGO problem for dimension $n=100$]{          
         \begin{minipage}[c]{.40\linewidth}
             \centering
             \includegraphics[width=1.0\textwidth]{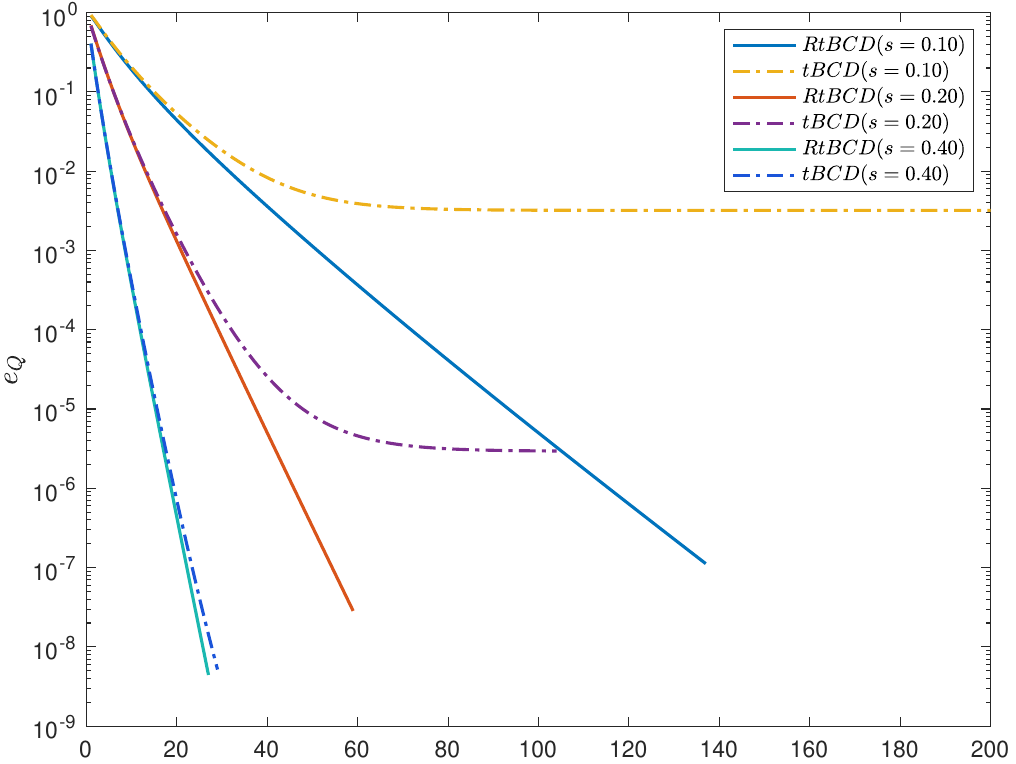}
         \end{minipage}
     }
     \quad
    \subfigure[PGO problem for dimension $n=100$ and observation rate $s=10\%$]{         
        \begin{minipage}[c]{.40\linewidth}
            \centering
            \includegraphics[width=1.0\textwidth]{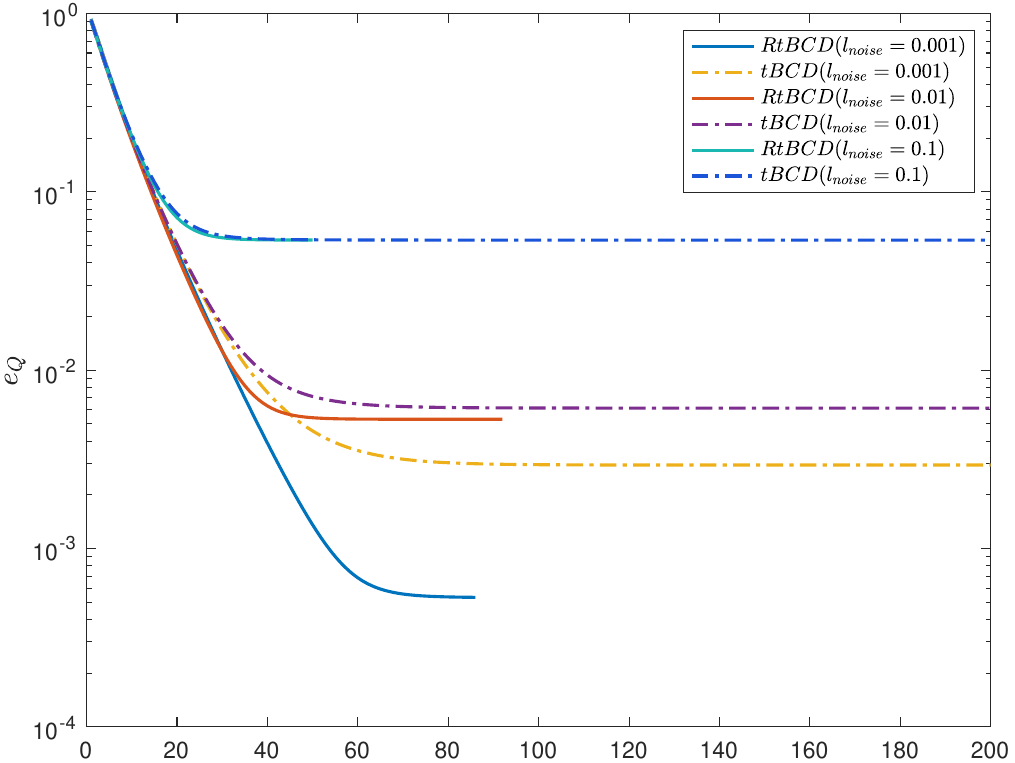}
        \end{minipage}
     }
     \caption{The impact of the parameter selection and the termination condition }
\label{fig1}
 \end{figure}

\begin{table}[htbp]
  \centering
  \caption{The impact of the parameter selection and the termination condition}\label{t1}
  \begin{tabular}{cccccccc}
    \hline
     \multicolumn{2}{c}{Algorithm}&\multicolumn{3}{c}{tBCD}& \multicolumn{3}{c}{RtBCD}\\
    $l_{\text{noise}}$ & $s$  & $e_{Q}$ & $n_{iter}$ & $time (s)$  & $e_{Q}$ & $n_{iter}$ & $time (s)$ \\
    \hline
    0 & 10$\%$  & 3.20e-3 & 200 & 2.87 & 1.12e-7 & 137 & 1.59 \\ 
    0 & 20$\%$  & 2.95e-6 & 104 & 9.26e-1 & 2.84e-8 & 59 & 4.80e-1 \\
    0 & 40$\%$  & 5.17e-9 & 29 & 2.34e-1 & 4.42e-9 & 27 & 2.35e-1 \\ 
    0.1 & 10$\%$  & 5.36e-2 & 200 & 2.33 & 5.37e-2 & 50 & 6.74e-1\\ 
    0.01 & 10$\%$ & 6.13e-3 & 200 & 2.09 & 5.31e-3 & 92 & 1.28\\
    0.001 & 10$\%$ & 2.94e-3 & 200 & 1.74 & 5.33e-4 & 86 & 1.26 \\ 
    \hline
  \end{tabular}
  \label{table4}
\end{table}

As shown in Figure \ref{fig1}, RtBCD exhibits near‑linear convergence in the noiseless case. As the observation rate decreases, the accuracy of tBCD degrades sharply, while RtBCD continues to achieve a steady, linear reduction in error. Under noisy conditions, RtBCD also attains lower final errors. Thanks to the newly added stagnation‐based stopping rule, RtBCD terminates early once the objective value ceases to change appreciably—again evident in Figure \ref{fig1}. As Table \ref{t1} demonstrates, RtBCD delivers higher accuracy, requires fewer iterations, and incurs shorter computation times compared to tBCD. These results validate our choice of parameter selection and the stagnation-based termination criterion, which not only accelerate convergence but also enhance solution quality.

Consequently, in all subsequent experiments we use RtBCD as the baseline for comparison rather than the original tBCD method. 

\subsection{Performance of improved two-block coordinate descent method}
In this subsection, we compare RtBCD, DBtBCD, and FtBCD on the PGO problem under varying noise levels and observation rates. Tables \ref{table5} and \ref{table6} report results for noise levels $l_{\text{noise}}\in \{0,0.1,0.05,0.01\}$ at $s=0.40$, $n=10$ and $s=0.10$, $n=100$, respectively. Table \ref{table7} shows performance at observation rates
$s\in\{0.3,0.4,0.5,0.6\}$ with $n=10$ and $l_{\text{noise}}=0.05$, and Table \ref{table8} shows performance at observation rates $s\in\{0.05,0.075,0.10,0.20\}$ with $n=100$ and $l_{\text{noise}}=0.05$. For RtBCD, the parameters are set as $k_{max}=500$, $\rho^{(0)}=0.01$, $\rho_1=1.0$, $d=2$, $\beta=10^{-4}$. For DBtBCD and FtBCD, the parameters are set as $k_{max}=500$, $\rho=0.01$, $d=2$, $\beta=10^{-4}$, $\delta=10^{-1}$. All experiments were repeated 100 times and the results were averaged.


\begin{table}[ht]
  \centering
  \caption{The numerical results of solving the PGO problem using algorithms RtBCD, DBtBCD, and FtBCD under different noise levels for given observation rate $s=0.40$ and dimension $n=10$}
  \resizebox{1.0\columnwidth}{!}{
  \begin{tabular}{ccccccccc}
    \hline
    Algorithm & $l_{\text{noise}}$  & $e_{Q}$ & $n_{iter}$ & $time (s)$ & $l_{\text{noise}}$ & $e_{Q}$ & $n_{iter}$ & $time (s)$ \\
    \hline
    \multirow{2}{*}{RtBCD}& 0.1 & 7.66e-2 & 24.49 & 1.40e-1 & 0.05 & 4.16e-2 & 24.40 & 1.60e-1 \\ 
    & 0.01 & 7.71e-3 & 25.70 & 1.44e-1  & 0 &  1.37e-5 & 36.38 & 1.75e-1  \\
    \multirow{2}{*}{DBtBCD} & 0.1 &  7.67e-2 & 16.89 & 8.58e-3 & 0.05 & 4.15e-2 & 14.81 & 7.07e-3 \\ 
    & 0.01 &  7.75e-3 & 11.21 & 2.37e-3   & 0 & 8.45e-6 & 15.68 & 2.94e-3 \\
    \multirow{2}{*}{FtBCD} & 0.1 &  7.23e-2 & 16.30 & 9.10e-3 & 0.05  & 3.91e-2 & 13.79 & 6.68e-3  \\ 
    & 0.01  & 7.30e-3 & 11.06 & 2.41e-3   & 0 & 9.37e-6 & 16.35 & 3.33e-3 \\
    \hline
  \end{tabular}}
  \label{table5}
\end{table}


\begin{table}[htbp]
  \centering
  \caption{The numerical results of solving the PGO problem using algorithms RtBCD, DBtBCD, and FtBCD under different noise levels for given $s=0.10$ and $n=100$}
  \resizebox{1.0\columnwidth}{!}{
  \begin{tabular}{ccccccccc}
    \hline
    Algorithm & $l_{\text{noise}}$  & $e_{Q}$ & $n_{iter}$ & $time (s)$ & $l_{\text{noise}}$ & $e_{Q}$ & $n_{iter}$ & $time (s)$ \\
    \hline
    \multirow{2}{*}{RtBCD}& 0.1 & 5.11e-2 & 59.37 & 7.22e-1 & 0.05  & 2.58e-2 & 64.58 & 7.64e-1 \\ 
    & 0.01  & 4.97e-3 & 71.17 & 6.14e-1 & 0 & 9.16e-6 & 116.39 & 1.11\\
    \multirow{2}{*}{DBtBCD} & 0.1 &  5.08e-2 & 34.67 & 1.74e-1 & 0.05 & 2.57e-2 & 32.17 & 1.28e-1 \\ 
    & 0.01 & 4.96e-3 & 24.63 & 7.56e-2  & 0 & 7.40e-6 & 41.82 & 1.45e-1 \\
    \multirow{2}{*}{FtBCD} & 0.1  & 4.68e-2 & 35.09 & 1.77e-1 & 0.05  & 2.36e-2 & 32.81 & 1.40e-1 \\ 
    & 0.01  & 4.56e-3 & 25.18 & 7.83e-2 & 0 &  7.68e-6 & 42.38 & 1.57e-1 \\
    \hline
  \end{tabular}}
  \label{table6}
\end{table}

It is worth noting that the algorithms are not always valid for solving PGO problems. We deem a run unsuccessful if the final error $e_{Q}$ exceeds $2\times l_{\text{noise}}$ (or $10^{-4}$ in the noise‑free case). In our 100-trial experiments at $s=0.40$ and $n=10$, RtBCD failed 4, 4, 2 and 2 times for noise level $l_{\text{noise}}=0.1,0.05,0.01,0$, respectively (Table \ref{table5}). It also failed 12, 3 times at $n=10$, $s=30\%$, and $40\%$, respectively (Table \ref{table7}), and 32, 16 and 1 times at $n=100$,  $s=5\%$, $7.5\%$ and  $10\%$,  respectively (Table \ref{table8}). In contrast, both DBtBCD and FtBCD succeeded in every trial, demonstrating a markedly higher success rate.

Across all scenarios with noise (Tables \ref{table5}–\ref{table8}), FtBCD consistently achieve lower error $e_{Q}$ than RtBCD and DBtBCD. Under noisy conditions, FtBCD reduces the error by $4.5\%\sim 8.5\%$ relative to RtBCD and DBtBCD. This demonstrates that employing the $F^*$‑norm formulation can indeed effectively enhance the accuracy of the solution. For $n=10$, DBtBCD and FtBCD require only $1.5\%\sim 6.5\%$ of the CPU time of RtBCD; while for $n=100$, it is $12\%\sim 30\%$. This highlights the superior efficiency of our algorithms. In particular, both DBtBCD and FtBCD converge in far fewer iterations than RtBCD, validating the effectiveness of our spectral initialization strategy.


\begin{table}[htpb]
  \centering
  \caption{The numerical results of solving the PGO problem using algorithms RtBCD, DBtBCD, and FtBCD at different observation rates for given $n=10$ and $l_{\text{noise}}=0.05$}
  \begin{tabular}{ccccccccc}
    \hline
    Algorithm & $s$  & $e_{Q}$ & $n_{iter}$ & $time (s)$ & $s$ & $e_{Q}$ & $n_{iter}$ & $time (s)$ \\
    \hline
    \multirow{2}{*}{RtBCD}& 30$\%$ & 4.54e-2 & 37.34 & 2.77e-1 & 40$\%$ & 3.94e-2 & 25.40 & 2.44e-1 \\ 
    & 50$\%$ & 3.48e-2 & 18.14 & 1.10e-1 & 60$\%$ & 3.10e-2 & 14.31 & 1.08e-1 \\
    \multirow{2}{*}{DBtBCD} & 30$\%$ & 4.59e-2 & 24.41 & 6.62e-3 &  40$\%$ & 3.93e-2 & 13.84 & 6.83e-3 \\ 
    & 50$\%$ & 3.47e-2 & 10.38 & 6.74e-3 & 60$\%$ & 3.10e-2 & 7.56 & 7.01e-3 \\
    \multirow{2}{*}{FtBCD} & 30$\%$ & 4.34e-2 & 24.67 & 7.22e-3 &  40$\%$ & 3.72e-2 & 13.86 & 7.16e-3 \\ 
    & 50$\%$ & 3.26e-2 & 9.67 & 6.63e-3 & 60$\%$ & 2.96e-2 & 7.34 & 6.88e-3 \\
    \hline
  \end{tabular}
  \label{table7}
\end{table}


\begin{table}[ht]
  \centering
  \caption{The numerical results of solving the PGO problem using algorithms RtBCD, DBtBCD, and FtBCD at different observation rates for given $n=100$ and $l_{\text{noise}}=0.05$}
  \resizebox{1.0\columnwidth}{!}{
  \begin{tabular}{ccccccccc}
    \hline
    Algorithm & $s$  & $e_{Q}$ & $n_{iter}$ & $time (s)$ & $s$ & $e_{Q}$ & $n_{iter}$ & $time (s)$ \\
    \hline
    \multirow{2}{*}{RtBCD}& 5$\%$ & 3.80e-2 & 179.40 & 1.98 & 7.5 $\%$ & 3.00e-2 & 94.14 & 1.21 \\
    & 10$\%$ & 2.53e-2 & 63.45 & 5.65e-1 & 20$\%$ & 1.72e-2 & 28.22 & 2.20e-1 \\ 
    \multirow{2}{*}{DBtBCD} &  5$\%$ & 3.80e-2 & 118.73 & 3.65e-1 & 7.5$\%$  & 3.00e-2 & 45.75 & 1.42e-1 \\
    & 10$\%$ & 2.53e-2 & 30.91 & 1.01e-1 & 20$\%$ & 1.71e-2 & 14.34 & 6.51e-2 \\ 
    \multirow{2}{*}{FtBCD}& 5$\%$  & 3.50e-2 & 118.86 & 3.82e-1 & 7.5$\%$ & 2.75e-2 & 46.54 & 1.51e-1 \\
    & 10$\%$ & 2.32e-2 & 31.56 & 1.02e-1 &  20$\%$ & 1.59e-2 & 14.43 & 6.39e-2 \\ 
    \hline
  \end{tabular}}
  \label{table8}
\end{table}

In general, our numerical experiments confirmed that the proper parameter selection and the stagnation-based stopping rule significantly increase both the accuracy and efficiency of the two-block coordinate descent method. The spectral initialization strategy further effectively reduces iterations and using the $F^*$-norm yields effective improvement in precision. Our algorithms exhibit superior accuracy, faster computation, and higher success rates, particularly at low observation rates. 

\section{Conclusions}\label{section7}
In this paper, we derive explicit solutions for the optimal rank‑k approximation of dual quaternion Hermitian matrices under both the $F$-norm and the $F^*$-norm. We then solve the PGO problem under the more appropriate and robust $F^*$-norm via the improved two-block coordinate descent method, leading to superior experimental accuracy. Our improved algorithm incorporates proper parameter selection, stagnation-based stopping criterion, and an efficient spectral initialization strategy. Extensive numerical experiments demonstrate that our approach achieves higher precision, faster computation, and higher success rates, especially in low-observation settings. 


\end{document}